\documentclass{amsart}
\usepackage[square,numbers]{natbib}

\usepackage[cmtip,all]{xy}
\usepackage{amsmath, amsthm, latexsym, amssymb, amsfonts}

\usepackage{alltt}
\usepackage[pdftex]{graphicx,color} 
\usepackage{algorithm}
\usepackage{algpseudocode}
\usepackage{setspace}
\usepackage{color,hyperref}
\usepackage{url}
\usepackage{longtable}
\usepackage[table]{xcolor}
\usepackage{mathtools}
\usepackage{tikz}
\usepackage{tikzscale}
\usepackage{pgfplots}
\pgfplotsset{compat=1.16} 
\usetikzlibrary{shapes}
\usepackage{ dsfont }
\usepackage{enumitem}
\usepackage{caption}
\usepackage{subcaption}
\usepackage{float}
\usepackage{bbm}

\newcommand{\supp}{\operatorname{supp}}

\newcommand{\reg}{\operatorname{reg}}

  \newcommand{\lcm}{\operatorname{lcm}}
\newcommand{\Floor}[1]{\left\lfloor #1 \right\rfloor}
\newcommand{\Floorfrac}[2]{\left\lfloor\frac{#1}{#2}\right\rfloor}
\newcommand{\Ceil}[1]{\left\lceil #1 \right\rceil}
\DeclareMathOperator{\red}{Red}
\DeclareMathOperator{\den}{den}
\DeclareMathOperator{\argmin}{argmin}
\DeclareMathOperator{\Span}{Span}
\newcommand{\set}[1]{\left\{#1\right\}}
\newcommand{\size}[1]{\left|#1\right|}
\newcommand{\calX}{\mathcal X}

\colorlet{darkred}{red!80!black}
\colorlet{darkblue}{blue!80!black}
\colorlet{darkgreen}{green!50!black}

\def\lm{\textrm{\small LM}}

\newcommand{\Q}{\mathbb Q}
\newcommand{\Z}{\mathbb Z}
\newcommand{\M}{\mathbb M}
\newcommand{\N}{\mathbb N}
\newcommand{\R}{\mathbb R}
\newcommand{\F}{\mathbb F}
\newcommand{\Pp}{\mathbb P}
\newcommand{\K}{{\mathbb K}}

\newcommand{\cl}[1]{\mathcal{#1}}

\newcommand{\la}{\langle}
\newcommand{\ra}{\rangle}

\def\a{{\bf a}}

\def\uu{{\bf u}}
\def\vv{{\bf v}}

\def\x{{\bf x}}

\def\ev{{\text{ev}}}

\newcommand{\cG}{\cl G}

\def\Tia{\Tilde{a}}
\def\Tid{\Tilde{d}}

\newtheorem{theorem}{Theorem}[section]

\newtheorem{lemma}[theorem]{Lemma}
\newtheorem{corollary}[theorem]{Corollary}

\newtheorem{proposition}[theorem]{Proposition}
\newtheorem{definition}[theorem]{Definition}

\newtheorem{example}[theorem]{Example}
\newtheorem{remark}[theorem]{Remark}

\numberwithin{equation}{section}

\renewcommand{\P}{{\mathbb P}}
\usepackage{cleveref}

\begin{document}

\title{Codes on Weighted
Projective Planes}


\author{Yağmur Çakıroğlu}
\author{Jade Nard\.{i}}
\author{Mesut \c{S}ah\.{i}n}
\address{ Department of Mathematics,
  Hacettepe  University,
  Ankara, TURKEY}
\address{Irmar, Ufr Mathématiques, Université de Rennes 1, Rennes, FRANCE}
\curraddr{}
\email{yagmur.cakiroglu@hacettepe.edu.tr}
\email{jade.nardi@univ-rennes.fr}
\email{mesut.sahin@hacettepe.edu.tr}
\thanks{The first and third authors are supported by T\"{U}B\.{I}TAK Project No:119F177. This article is part of the first author's Ph.D. thesis under the supervision of the third author. The second author is supported  by the French National Research Agency through ANR \textit{Barracuda} (ANR-21-CE39-0009) and the French government \textit{Investissements d’Avenir} program ANR-11-LABX-0020-01.}


\date{}


\begin{abstract} We comprehensively study weighted projective Reed-Muller (WPRM) codes on weighted projective planes $\P(1,a,b)$. We provide the universal Gr\"obner basis for the vanishing ideal of the set $Y$ of $\F_q$--rational points of $\P(1,a,b)$ to get the dimension of the code. We determine the regularity set of $Y$ using a novel combinatorial approach. We employ footprint techniques to compute the minimum distance.
\end{abstract}
\maketitle 

\section{Introduction}

Let $\F_q$ be the finite field with $q$ elements and $\K$ be its algebraic closure. We study linear codes from weighted projective planes $\P(1,a,b)$ for two positive integers $a\leq b$.  Recall that the \emph{weighted projective plane} is the quotient space
\[\P(1,a,b)=(\K^{3}\setminus \{0\})/\K^*\] under the following equivalence relation: for every $(x_0,x_1,x_2) \in \K^3$,
\begin{equation*}
(x_0,x_1,x_2)\sim(\lambda^{}x_0,\lambda^{a}x_1,\lambda^{b}x_2) \mbox{ for } \lambda \in \K^*.
\end{equation*}
We can assume without loss of generality that the integers $a$ and $b$ are coprime as $\P(1,a,b) \simeq \P(1,ac,bc)$ for every positive integer $c$ (see \cite[Lemma 1.1]{Dol81}).

It is well known that the set $\P(1,a,b)(\F_q)$  of $\F_q$-rational points consists of equivalence classes $[x_0:x_1:x_2]$ having representatives with all coordinates $x_0,\: x_1, \: x_2$ lying in $\F_q$, see \cite[Lemmas 6 and 7]{Per03}. More precisely, one can choose representatives $[1:y_1:y_2]$ where $(y_1,y_2) \in \F_q^2$, together with $[0:y_1:y_2]$ as given in Remark~\ref{rk:rat_pts}.

The weighted projective plane $\Pp(1,a,b)$ is a simplicial toric variety. It thus comes with the polynomial ring $S=\F_q[x_0,x_1,x_2]$, which is graded via $\deg x_0 =1$ $\deg x_1 =a$ and $\deg x_2 =b$. A degree $d\in \N$ defines a polygon 
\begin{equation}\label{eq:Pd}
P_d:=\set{(x,y)\in \R^2 \: : \: x\geq 0, \: y \geq 0, \: ax+by \leq d}
\end{equation}
and the integral points of $P_d$ give rise to monomials of degree $d$: 
\begin{equation}\label{eq:Md}
\M_d:=\set{\x^{\a,d}:=x_0^{d-aa_1-ba_2}x_1^{a_1}x_2^{a_2} :  \a=(a_1,a_2)\in P_d \cap \Z^2}.
\end{equation}
The polynomial ring $S$ can thus be written as follows.
\begin{equation*}\label{homgrdring}
S= \bigoplus_{d\geq 0 } S_d \text{ where } S_d=\Span \M_d.
\end{equation*}

In order to obtain a linear code, we evaluate homogeneous polynomials of degree $d \geq 1$ at a subset $Y=\{P_1,\dots,P_n\}\subseteq \P(1,a,b)(\F_q)$ of $\F_q$-rational points. As points of $\P(1,a,b)$ are orbits, we have to choose representatives for each of the points: given a point $P \in \P(1,a,b)(\F_q)$, we choose a representative triple $(y_0,y_1,y_2)\in\F_q^3$ and we define $f(P)=f(y_0,y_1,y_2)$ for any $f \in S_d$. This defines the following evaluation map
\begin{align}\label{evmap}
\ev_Y&:\begin{array}{rcl}
S_d&\rightarrow& \F_q^{n}\\
f&\mapsto&\left(f(P_1),\dots,f(P_n)\right)
\end{array}
\end{align}
whose image, denoted by $C_{d,Y}$, is the evaluation code at $Y$ of degree $d$.
A different choice of representatives gives monomially equivalent code, \textit{i.e.} each coordinate is multiplied by a non-zero element of $\F_q$.

We write $[n,k,d_{min}]_q$ for the main parameters of the code, which we recall next. The \emph{length} $n$ is the cardinality $|Y|$ of the set $Y$. The \emph{dimension} $k$ is the dimension of $C_{d,Y}$ as an $\F_q$--vector space. The \emph{Hamming weight} $\omega(c)$ of a codeword $c=\ev_Y(f)=(f(P_1),\dots,f(P_n))\in C_{d,Y}$ is the number of non--zero components of $c$ which is also $n-|V_Y(f)|$, where $V_Y(f):=\{P\in Y: f(P)=0\}$. Finally, the \emph{minimum distance} $d_{min}$ of $C_{d,Y}$ is the minimum weight among all codewords $c\in C_{d,Y}\backslash\{0\}$.

If $Y$ consists only of affine points, \textit{i.e.} of the form $(1,x,y)$ for $(x,y) \in \F_q^2$, then the code $C_{d,Y}$ is called the weighted (affine) Reed-Muller (WRM or WARM) code of $(a,b)$--weighted degree $d$. Its length is $n=\size{\F_q^2}=q^2$. The dimension and the minimum distance have been given by S\o rensen \cite[Theorem~1]{S1992}. 

If $Y$ is the whole set of $\F_q$-rational points of the weighted projective plane $\P(1,a,b)$, then the code is simply denoted by $C_{d}$ and called a weighted projective Reed-Muller (WPRM) code. Its length is $n=\size{\P(1,a,b)(\F_q)}=q^2+q+1$ for any integers $a$ and $b$. The dimension and the minimum distance were given explicitly by Aubry et al. in \cite{ACGLOR2017} only for small degrees $d\leq q$ which are multiple of both $a$ and $b$. Recently, \cite[Theorem 5.1]{AP24} relaxed the divisibility hypothesis but still gave the minimum distance when $d \leq q$.  Let us mention that S{\o}rensen \cite{S1992} also introduced some different codes bearing the name ``weighted projective Reed-Muller codes'', where the evaluation is done for a proper subspace of $\bigoplus_{d' \leq d} S_{d'}$ at the $\F_q$-rational points of the \emph{classical} projective plane $\P^2$.

In the present paper, we deal with the case $Y=\P(1,a,b)(\F_q)$ with arbitrary degree $d$, thus extending and generalizing the results of the paper by \c{C}ak\i ro\u{g}lu and \c{S}ahin \cite{CS2024} given for the case where $a=1$.  We leverage the fact that weighted projective spaces are toric varieties and we follow the strategy described in \cite{JNPro22}, which connects the combinatorics of the polygon $P_d$ (see Equation~\ref{eq:Pd}) and the parameters of the code $C_d$. 

\medskip

The paper is organized as follows. Section~\ref{sec:dim} displays the universal Gr\"obner basis for the vanishing ideal of $Y=\P(1,a,b)(\F_q)$ and the projective reduction of the polygon $P_d$. This provides a basis, and thus the dimension, of $C_d$ for any degree $d \geq 1$. In Section~\ref{sec:reg}, we determine the regularity set of $Y$ which helps eliminate the trivial codes as well as giving a lower bound for the minimum distance. In Section~\ref{sec:min-dist} we employ footprint techniques to bound the minimum distance from below. Candidates for minimal weight codewords are given in Section \ref{sec:poly_weight}. Our results for the minimum distance are summarized in Section \ref{sec:conclusion}, followed by some refinements in the cases for which we are only able to determine a lower bound the minimum distance.

\section{Dimension}\label{sec:dim}

In this section, our goal is to give a closed formula for the dimension of $C_d$. As an evaluation code, the code $C_d$ is isomorphic to the quotient of $S_d$ by the degree--$d$ part of the vanishing ideal of $Y$. In the case we are interested in, \textit{i.e}. $Y=\P(1,a,b)(\F_q)$, we know a generating set of the vanishing ideal of $Y$, which we prove to be a universal Gr\"obner basis. This allows us to identify a combinatorial set called a \emph{projective reduction} that corresponds to the standard monomials of degree $d$, giving rise to the basis of the code $C_d$.

\subsection{The vanishing ideal of the set of the rational points of $Y=\Pp(1,a,b)$}
In this subsection, we give a unique (up to a constant factor) minimal generating set for the vanishing ideal of $I(Y)$, which is both the universal Gr\"obner basis and the Graver basis.

An ideal generated by binomials $\x^{\uu}-\lambda\x^{\vv}$, with $\lambda \in \K \setminus \set{0}$, is called a \emph{binomial ideal}.  
The following reveals that the ideal $I(\P(1,a,b)(\F_q))$ is \emph{pure binomial}, \textit{i.e.} generated by pure difference binomials $\x^{\uu}-\x^{\vv}$.
 
\begin{theorem}\cite[Corollary 5.8]{MS2022}\label{thm:ideal} The vanishing ideal $I(\P(1,a,b)(\F_q))$ of the set of $\F_q$--points of $\P(1,a,b)$ is generated by the following binomials.
\begin{align*}
    f_0&=x_{1}x_{2}\left(x_2^{(q-1)a}-x_{1}^{(q-1)b}\right)\\
    f_1&=x_0x_{2}\left(x_2^{q-1}-x_0^{(q-1)b}\right) \\
    f_2&=x_0x_{1}\left(x_{1}^{q-1}-x_0^{(q-1)a}\right)
\end{align*}
\end{theorem}

\begin{remark}The polynomials $f_i$ are numbered so that $f_i$ does not involve the variable $x_i$ for every $i \in \set{0,1,2}$.

It is clear that the monomials of $f_0,f_1,f_2$ cannot divide each other and so $|G(M_{I(Y)} )| = 6$ in the notation of \cite[Corollary 3.6]{CTV2016}. Thus, the binomials $f_0,f_1,f_2$ are \textit{indispensable}, \textit{i.e.} they appear (up to a nonzero constant) in every minimal binomial generating set of $I(Y)$. In other words, (up to a nonzero constant) $\{f_0,f_1,f_2\}$ is the unique minimal generating set.
\end{remark}

 We would like to obtain the \textit{universal} Gr\"obner basis of $I(Y)$ which is a Gr\"obner basis with respect to any monomial ordering. Since there are only finitely many distinct reduced Gr\"obner bases, their union is the universal Gr\"obner basis. To accomplish this goal, we appeal to \cite[Proposition 4.2]{CTV2016}, saying that a binomial in the universal Gr\"obner basis must be \textit{primitive,} which we define next.

\begin{definition}{\cite[Definition 4.1]{CTV2016}} A nonzero binomial $\x^{\uu}-\x^{\vv}:=x_0^{u_0}x_1^{u_1}x_2^{u_2}-x_0^{v_0}x_1^{v_1}x_2^{v_2}$ in a pure binomial ideal $I$ is called a \emph{primitive binomial} of $I$ if there exists no other binomial $\x^{\uu'}-\x^{\vv'}\in I\setminus \{0\}$ such that $\x^{\uu'}$
divides $\x^{\uu}$ and $\x^{\vv'}$
divides $\x^{\vv}$. The
set of all primitive binomials of $I$ is called the \emph{Graver basis} of $I$.
\end{definition}

We are ready to give the main result of this subsection showcasing a rare instance where an ideal has a unique minimal generating set which is both the universal Gr\"obner basis and the Graver basis. This generalizes \cite[Theorem 2.8]{BDG2019} which gives a reduced Gr\"obner basis which is also a universal Gr\"obner basis for the vanishing ideal of $Y=\Pp^m(\F_q)$ over an algebraic extension of $\F_q$.
\begin{theorem} \label{thm:UniversalBasis}
The generating set $\{f_0,f_1,f_2\}$ is both the universal Gr\"obner basis and the Graver basis of $I(Y)$. 
\end{theorem}

\begin{proof} Take a homogeneous binomial $\x^{\uu}-\x^{\vv}$ which is an element of the universal Gr\"obner basis of $I(Y)$. By \cite[Proposition 4.2]{CTV2016}, it must be primitive. Thus, it suffices to show that there is no primitive homogeneous binomial other than $f_0,f_1,f_2$.

By \cite[Proposition 5.6]{MS2022} and the proof of \cite[Theorem 3.7]{MS2022}, a homogeneous binomial in $I(Y)$ is of the form
\begin{equation*}
\x^{\uu}-\x^{\vv}=\x^{\a}(\x^{m^+}-\x^{m^-}) \text{ with } \supp(\x^{m^+})\cap \supp(\x^{m^-})=\emptyset, 
\end{equation*}
where $\supp(\x^{\a}):=\{j\in \set{0,1,2}: x_j \mid \x^{\a}\}$ and $m^+-m^-\in (q-1)L_{\beta(\varepsilon)}$. Recall that $\beta=\begin{pmatrix} 1 &  a & b \end{pmatrix}$, $\beta(\varepsilon)$ is the submatrix of the matrix $\beta$ with the columns $\beta_{j+1}$ where $j\in \varepsilon=\supp(\x^{\a}) \subseteq \{0,1,2\}$ and   $L_{\beta(\varepsilon)}$ is the lattice, which is the integer points of the kernel of the linear map represented by the matrix $\beta(\varepsilon)$.

Suppose that $\varepsilon=\set{0,1,2}$. Then, without loss of generality, we may assume that $\x^{\uu}-\x^{\vv}=x_0^{a_0}x_1^{a_1}x_2^{a_2}(x_2^{m^+_2}-x_0^{m^-_0}x_1^{m^-_1})$, with positive $a_0,a_1,a_2$. Then, 
\[\x^{\uu}-\x^{\vv}=x_0^{a_0-1}x_1^{a_1-1}x_2^{a_2-1}(x_0x_1x_2^{m_2^++1} - x_0^{m^-_0+1}x_1^{m^-_1+1}x_2).
\]
Thus, $\x^{\uu}-\x^{\vv}$ is not primitive, since $\x^{\uu'}-\x^{\vv'}:=x_0x_1x_2^{m_2^++1} - x_0^{m^-_0+1}x_1^{m^-_1+1}x_2$ belongs to $ I(Y)$ and we have both $\x^{\uu'}$
divides $\x^{\uu}$ and $\x^{\vv'}$
divides $\x^{\vv}$.

Since $L_{\beta(\varepsilon)}=\{0\}$, when $\size{\varepsilon}=1$, we just need to consider the case $\size{\varepsilon}=2$. If $\varepsilon=\{1,2\}$ then $\beta(\varepsilon)=[ a \: b]$, $L_{\beta(\varepsilon)}$ is the sublattice of $\Z^2$ spanned by $(b,-a)$ and we have $\x^{\uu}-\x^{\vv}=x_1^{a_1}x_2^{a_2}(x_1^{m^+_1}-x_2^{m^-_2})$ for some positive integers $a_1,a_2$ and $(m^+_1,0)-(0,m^-_2)=(q-1)l(b,-a)$ with a positive integer $l$. Thus, the monomials $x_1x_2x_1^{(q-1)b} \mid \x^{\uu}$ and $x_1x_2x_2^{(q-1)a} \mid \x^{\vv}$. The only primitive binomial with these properties is clearly $f_0$. The other two cases are done similarly.
\end{proof}

\subsection{Projective reduction}
We want to rely on the notion of \emph{projective reduction} of the polygon $P_d$ introduced in \cite[\textsection 3.1]{JNPro22} to get canonical representatives for the cosets in $S_d/I(Y)_d$ through the identification of lattice points of $P_d$ with monomials in $S_d$. We follow the notations introduced in \cite{JNPro22}, which generalizes the notion of projective reduction for the classical projective spaces given in \cite[\textsection 2.1]{BDG2019}. 

\begin{remark}
Remark~3.6 of \cite{JNPro22} assumes that the polytope is integral and claims that the \emph{projective reduction} is only relevant when the polytope has the same normal fan as the variety.
The latter condition is always satisfied on $\P(1,a,b)$ for any $P_d$. But, when $a$ (resp. $b$) does not divide $d$, the vertex $(d/a,0)$ (resp. $(0,d/b)$) of $P_d$ is not integral. This does not cause a problem because a vertex corresponds to a monomial involving only one variable, hence it cannot be equivalent to another monomial modulo $I(Y)$. In other words, the cosets in $S_d/I(Y)_d$ of these monomials are singletons. We just need to work out the other cosets, which can be done by the methods of \cite{JNPro22}. The reader is invited to consult \cite[Lemma 5.15 and Theorem 5.16]{MS2022} for more general toric varieties.
\end{remark}

Recall some definitions about lattice polygons. Let $P$ be a convex lattice polygon. We denote a face $Q$ of $P$ by $Q\prec P$ and define its \emph{interior} to be the set of points not lying on any proper face, \textit{i.e.}

\begin{equation*}\label{eq:interior}
P^{\circ}=P \setminus \bigcup_{\substack{Q\prec P \\Q \neq P}} Q.
\end{equation*}

By \cite[Theorem 3.5]{JNPro22}, a basis for the code  $C_d$ is given by certain integral points of the polygon $P_d$, that are chosen with respect to the following equivalence relation.

\begin{definition}\cite[Definition 3.4]{JNPro22} Given a lattice polytope $P$, we define an equivalence relation $\sim_{P}$ on the set of its lattice points $P\cap \Z^{N}$ by
\[m\sim_{P} m' \iff \exists Q\prec P \mbox{ such that } m,m'\in Q^{\circ} \mbox{ and } m-m'\in (q-1)\Z^{N}\] where $Q^{\circ}$ is the interior of $Q$.
A \textbf{projective reduction} $\red{P}$ of $P$ is defined to be a set of representatives of elements of $P\cap \Z^{N}$ modulo $\sim_{P}.$
\end{definition}
 
 As suggested by \cite[Definition 4.3]{JNPro22}, we consider a particular projective reduction of $P_d$ coming from a monomial ordering. Throughout this paper, we will deal with the lex ordering with $x_2>x_1>x_0$: a monomial $x^{\a,d}=x_2^{a_2}x_1^{a_1}x_0^{d-aa_1-ba_2}$ of degree $d$ is smaller than a monomial $x^{\a',d'}=x_2^{a'_2}x_1^{a'_1}x_0^{d'-aa'_1-ba'_2}$  of degree $d'$ if and only if the left most non-zero number in $(a_2-a'_2,a_1-a'_1,a(a'_1-a_1)+b(a'_2-a_2))$ is negative. In this case, we write $x^{\a,d} <_{\text{lex}} x^{\a',d'}$.

 Note that the left most non-zero number in $(a_2-a'_2,a_1-a'_1,a(a'_1-a_1)+b(a'_2-a_2))$ being negative is equivalent to the left most non-zero number in $(a_2-a'_2,a_1-a'_1)$ being negative, \textit{i.e.} $(a_1,a_2) <_{\text{lex}} (a'_1,a'_2)$. 
 
We denote by $\red(d)$ the projective reduction with respect to the aforementioned ordering, see Figure \ref{subfig:redc}, \textit{i.e.}
\begin{align}\label{eq:def_red(d)}
\red(d)&=\set{\min_{\text{lex}}\set{m \in \theta} : \theta \in (P_d \cap \Z^2)/ \sim_{P_d}}\\\label{eq:def_red(d)_faces}
&=\bigcup_{Q\prec P} \set{\min_{\text{lex}}\set{m \in \theta} : \theta \in (Q^\circ \cap \Z^2)/ (q-1)\Z^2}.
\end{align}

\begin{remark}
    If $a=b=1$, $\red(d)$ identifies with the exponents of the projective reduction of monomials of degree $d$ defined in \cite[Definition 2.1]{BDG2019}.
\end{remark}

\begin{figure}[htb]
\begin{subfigure}[b]{0.32\textwidth}
	\centering
		\includegraphics[keepaspectratio]{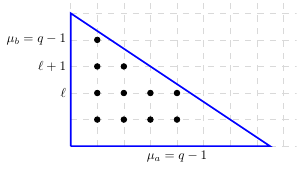}
		\caption{Reduction of $P^{\circ}$}\label{subfig:reda}
	\end{subfigure}
	\begin{subfigure}[b]{0.32\textwidth}
		\centering
		\includegraphics[keepaspectratio]{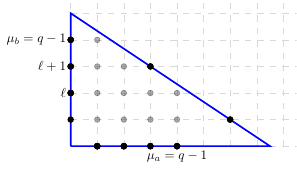}
		\caption{Reduction of edges}\label{subfig:redb}
	\end{subfigure}
	\begin{subfigure}[b]{0.32\textwidth}
	\centering
	\includegraphics[keepaspectratio]{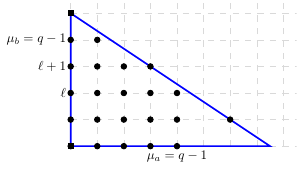}
	\caption{Total reduction of $P$}\label{subfig:redc}
\end{subfigure}
	\caption{Reduction of a polygon associated to $\P(1,2,3)$ for $q=5$ and $d=15$.}
\end{figure}

\begin{theorem}{\cite[Theorem 3.5]{JNPro22}}\label{theo:basis_code}
    A basis for the code $C_d$ is given by the images under the evaluation map $\ev_Y$ of the monomials in $\overline{\M}_{d}:=\{\x^{\a,d} : \a\in \red(d)\}$. Therefore $k=\dim_{\F_q} C_d=|\red(d)|$. 
\end{theorem}

We introduce the following notation that will be useful to compute the dimension of the code:
\begin{align}
\mu_{a}&=\min\set{\Floorfrac{d-1}{a},q-1}, &E_x(d)=\set{(x,0) : 0 \leq x \leq \mu_a}\\
\mu_{b}&=\min\set{\Floorfrac{d-1}{b},q-1}, & E_y(d)=\set{(0,y) : 0 \leq y \leq \mu_b}.
\end{align}

If $N_0=(x',y')$ is the first interior lattice point on the hypotenuse of the triangle $P_d$ (the one with the smallest positive $y$-coordinate and with $x'>0$), set
\begin{equation}\label{eq:def_t}
    t=\min\set{q-2,\Floorfrac{d-by'-1}{ab}}
\end{equation}
and define
\begin{equation}\label{eq:Eh}
E_h(d)=\set{(x'-ib,y'+ia)\in \N^2 :  0\leq i \leq t },
\end{equation}
and if there is no such point $N_0$, we let $E_h(d)=\emptyset$.



\begin{remark} \label{rem:cases}
The numbers $\alpha_2=\Floor{\frac{d-1-a(q-1)}{b}}$ and $\ell = \max\set{0,\min\set{q-1,\alpha_2}}$ will play an important role for the rest of the paper. The integer $\ell$ is the biggest integer $y\in [1,\mu_{b}]$ satisfying the inequality $a(q-1)+by \leq d-1$, \textit{i.e.} such that the lattice point $(q-1,y)$ lies in $P_d^\circ$. Notice that we have
\[
\left\{\begin{array}{ccll}
    \mu_a=\lfloor \frac{d-1}{a}\rfloor,&  \alpha_2<0& \text{ and }  \ell=0 & \text{if }  d\le a(q-1), \\
    \mu_a=q-1,& \alpha_2=0& \text{ and }  \ell=0 &   \text{if }  a(q-1)<d\le a(q-1)+b,\\
    \mu_a=q-1,& 1 \leq  \alpha_2\leq q-2& \text{ and }  \ell=\alpha_2  &  \text{if } 
   a(q-1)+b<d \leq (q-1)(a+b),\\
    \mu_a=q-1,& q-1 \leq \alpha_2& \text{ and }  \ell=q-1  &  \text{if } 
   (q-1)(a+b)<d.
\end{array}\right.
\]
 
\end{remark}

\begin{lemma}\label{lem:red} Let $\ell$ be as in Remark \ref{rem:cases}. Then, the set $\red(d)$ corresponding to the lex ordering with $x_2>x_1>x_0$ is the union $$R(d)\cup T(d) \cup H(d), \text{ where}$$
   \begin{align*}
 R(d)&=\set{(x,y) \in \Z^2: 0\leq x \leq \mu_a \text{ and } 0\leq y \leq \ell},\\
 T(d)&=\set{(x,y) \in \Z^2: 0\leq x \leq \Floor{\frac{d-1-by}{a}} \text{ and } \ell+1\leq y \leq \mu_b},\\
 H(d) & =E_h(d)\cup \left[\set{(0,d/b),(d/a,0)}\cap \Z^2\right].
\end{align*}  
\end{lemma}
\begin{proof} In order to determine elements of $\red(d)$ as described in Equation \eqref{eq:def_red(d)_faces}, we run through the faces $Q\prec P_d$ and we choose the ``smallest" non--equivalent points $\a\in Q^{\circ}$ with respect to the aforementioned lex ordering.

\medskip

First of all, the vertex  $(0,0)$ lies in the set $R(d)$. The vertices $(0,\frac{d}{b})$ and $(\frac{d}{a},0)$ clearly belong to the set $H(d)$ when $b|d$ or $a|d$, respectively. 

\smallskip

Now let us deal with the edges of $P_d$. If $Q$ is the edge of $P_d$ on $x$-axis, then the lattice points in the interior $Q^\circ$ of $Q$ are of the form $(x,0)$ for $1\leq x \leq \Floorfrac{d-1}{a}$. By definition of the reduction, one can easily check that
\[Q^{\circ} \cap \red(d)=\set{(x,0) : 1 \leq x \leq \mu_a} = E_x(d) \subseteq R(d).\]

If $Q$ is the edge of $P_d$ on $y$-axis, then 
\[Q^{\circ} \cap \red(d)=\set{(0,y) : 1 \leq y \leq \mu_b} = E_y(d) \subseteq R(d) \cup T(d).\]

Let $Q$ be the edge of $P_d$ lying on the line $ax+by=d$. If $N_0=(x',y')$ is the smallest lattice point on $Q^{\circ}$ (the one with the smallest positive $y$-coordinate) then all the lattice points on $Q^{\circ}$ are of the form $N_i=(x'-ib,y'+ia)$ for some non-negative integer $i$. 

Note that the biggest lattice point on $Q^{\circ}$ (the one with the biggest $y$-coordinate and positive $x$-coordinate on the line $ax+by=d$) is $N_{i_{max}}=(x'-i_{max} b,y'+i_{max} a)$ where $i_{max}=\Floorfrac{d-by'-1}{ab}$. Indeed, by the definition of the floor part, the coordinates of $N_{i_{max}}$ satisfy
\begin{align*}
    x'-i_{max} b & \in \left[\frac{1}{a},\frac{1}{a} + b\right) ,\\
    y'+i_{max} a &\in \left(\frac{d-1}{b}-a,\frac{d-1}{b}\right].
\end{align*}%
Moreover, for any nonnegative integers $i$ and $j$, we have $N_i-N_{j}=(i-j)(-b,a)$. As $a$ and $b$ are coprime, the points $N_i$ and $N_j$ are equivalent if and only if $q-1$ divides $i-j$. Therefore, since $t=\min\set{q-2,i_{max}}$ (see Equation \eqref{eq:def_t}), the reduction $Q^{\circ} \cap \red(d)$ is exactly $E_h(d)$.

\smallskip

Finally, let $Q$ be $P_d$ itself. With the same reasoning as for the horizontal and vertical edges, we can deduce that the reduction of $Q^{\circ}$ lies in the square $[1,\mu_a] \times [1,\mu_b]$. More precisely, we have 
\[Q^{\circ} \cap \red(d) = P_d^{\circ} \cap \left( [1,\mu_a] \times [1,\mu_b]\right).\]

Let us make the intersection $P_d^\circ \cap \left([1,\mu_a] \times [1,\mu_b] \right)$ more explicit by looking at each horizontal line $y=y_0$ for $1\leq y_0 \leq \mu_b$.

On the line $y=y_0$, notice that $\Floorfrac{d-1-by_0}{a}$ is the $x$-coordinate of the right most lattice point of $P_d^{\circ}$. Then 
\begin{equation}\label{eq:proof_red_interior}
P_d^{\circ} \cap \red(d) = \bigcup_{1 \leq y \leq \mu_b} \set{(x,y) : 1 \leq x \leq \min\set{q-1,\Floorfrac{d-1-by}{a}}}.
\end{equation}

\begin{itemize}
    \item If $\ell = 0$, then $d-1 < a(q-1)+b$ by Remark \ref{rem:cases}. Then for every $y\geq 1$, we have $\lfloor \frac{d-1-by}{a}\rfloor<\Floorfrac{d-1-b}{a}\leq q-1$.  The set in Equation \eqref{eq:proof_red_interior} is simply
    \[P_d^{\circ} \cap \red(d) = \bigcup_{1 \leq y \leq \mu_b} \set{(x,y) : 1 \leq x \leq \Floorfrac{d-1-by}{a}} \subset T(d).\]

    In this case, the whole reduction $\red(d)$ is the union of 
    \begin{itemize}
        \item $R(d)=\set{(0,0)} \cup E_x(d)$ which contains the reductions of the vertex $(0,0)$ and of the horizontal edge,
        \item $T(d)=\{(x,y) : 0\leq x \leq \Floor{\frac{d-1-by}{a}} \text{ and } 1\leq y \leq \mu_b\}$, which consists of the reduction of the vertical edge $E_y(d)$ and of the interior of $P_d$, 
        \item $H(d)$, the lattice points among the last two vertices and the reduction of the interior of the edge lying on the line $ax+by=d$.
    \end{itemize}

    \item If $\ell \geq 1$, then by the virtue of Remark \ref{rem:cases}, we can rewrite  Equation \eqref{eq:proof_red_interior} as follows. 
\begin{align*}P_d^{\circ} \cap \red(d) =& \bigcup_{1 \leq y \leq \ell} \set{(x,y) : 1 \leq x \leq q-1=\mu_a} \\&\cup \bigcup_{\ell+1 \leq y \leq \mu_b} \set{(x,y) : 1 \leq x \leq \Floorfrac{d-1-by}{a}}.
\end{align*}
The first union lies in the rectangular area $R(d)$, whereas the second union lies in the trapeze area $T(d)$. 
\end{itemize}

In both cases, the description given via Equation \eqref{eq:def_red(d)_faces} agrees with the one given in the statement, which completes the proof. 
\end{proof}

\begin{example} \label{ex:P(1,2,3)_5} Let $q=5$, $X=\P(1,2,3)$ and $Y=X(\F_q)$. 

Firstly let $d=7$. We count the number of points of $\red(d)$ in Figure \ref{subfig:a}. Clearly, $\ell=0$ and $R(d)=\{(0,0),(1,0),(2,0),(3,0)\}$. So, $\size{R(d)}=4$. Similarly, we have $T(d)=\{(0,1),(1,1),(0,2)\}$. Thus, $|T(d)|=3.$ Also, since $H(d)=\{(2,1)\}$ we get $|H(d)|=1$. Therefore, $\dim_{\F_5}(C_{7,Y})=\size{\red(7)}=4+3+1=8.$

\begin{figure}[h]
\begin{subfigure}[b]{0.45\textwidth}
\centering
\includegraphics[keepaspectratio]{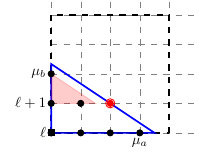}
\caption{$\red(7)=P_7 \cap \Z^2$ with $\ell=0$,
\\ \textit{$\mu_a=3$ and $\mu_b=2$.}} \label{subfig:a}
\end{subfigure}
\begin{subfigure}[b]{0.45\textwidth}
\centering
\includegraphics[keepaspectratio]{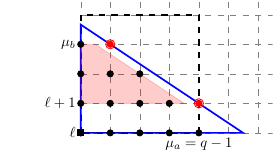}
\caption{$\red(11)$ with $\ell=0$, \\
$\mu_a=4$ and $\mu_b=3$}\label{subfig:b}
\end{subfigure}
\caption{The set $\red(d)$ for $d=7,11$ with $(q,a,b)=(5,2,3).$}
\end{figure}
Figure \ref{subfig:b} shows $|R(11)|=5$, $|T(11)|=8$, $|H(11)|=2$ and
$\dim_{\F_5}(C_{11,Y})=15.$

\begin{figure}[h]
\begin{subfigure}[b]{0.45\textwidth}
\centering
\includegraphics[keepaspectratio]{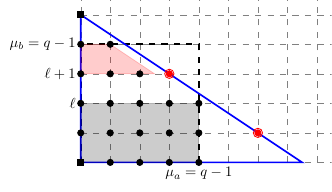}
\caption{$\mu_a=\mu_b=q-1$ and $T(15)\neq \varnothing$ }\label{subfig:c}
\end{subfigure}
\begin{subfigure}[b]{0.45\textwidth}
\centering
\includegraphics[keepaspectratio]{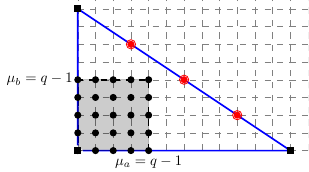}
\caption{$\ell=\mu_a=\mu_b=q-1$ and $T(24)= \varnothing$ }\label{subfig:d}
\end{subfigure}
\caption{The set $\red(d)$ for $d=15,24$ with $(q,a,b)=(5,2,3)$.
}
\label{fig:dimq5}
\end{figure}

Figure \ref{subfig:c} reveals similarly that
$\dim_{\F_5}(C_{15,Y})=15+5+3=23.$

Finally, Figure  \ref{subfig:d} proves that
$\dim_{\F_5}(C_{24,Y})=25+5=30.$

\end{example}

Example \ref{ex:P(1,2,3)_5} illustrates how the dimension of the code $C_d$ can be computed combinatorially by using Theorem \ref{theo:basis_code} and Lemma \ref{lem:red}.

In order to express the cardinality of the set $H(d)$, we appeal to a very classical function in the theory of numerical semigroups.
\begin{definition}\cite{A2005} 
Let $a,b$ and $d$ be positive integers. The \textbf{denumerant} function $\den(d;a,b)$ is defined as the number of non-negative integer representations of $d$ by $a$ and $b$, that is, the number of solutions $(m_a,m_b) \in \N^2$  of 
\[d=m_a a+m_b b.\] 
\end{definition}
The quantity $den(d;a,b)$ is positive if and only if $d$ belongs to the semigroup $\langle a,b \rangle_{\N}$ generated by $a$ and $b$ over $\N$. It is known that (see \cite[Chapter 4]{A2005} or \cite{SertozOzluk}) if $d=\lambda ab+s$ with $0\le s <ab$ then $den(d;a,b)=\lambda +den(s;a,b)$ with ${den(s;a,b) \in \{0,1\}}$. More precisely, we have
\begin{equation*}
\den(s;a,b)=\begin{cases}
0 \text{ or } 1 & \text{if } 0<s<ab\\
1 & \text{for all } ab-a-b<s<ab\\
0 & \text{if } s=ab-a-b.
\end{cases}
\end{equation*}

\begin{theorem}\label{theo:dimab} Let $Y=\P(1,a,b)(\F_q)$ for two relatively prime positive integers $a\leq b$. Let $d \geq 1$ and let $\ell$ be as in Remark \ref{rem:cases}. We write $\mathbf{1}_{a \mid d}$ (resp. $\mathbf{1}_{b \mid d}$) for the integer being $1$ when $a$ (resp. $b$) divides $d$, and $0$ otherwise.

Then, the dimension of the code $C_d$ is given by
\begin{equation}\label{eq:dimab} \dim_{\F_q}(C_d) = 
(\ell+1)\mu_a+\mu_b+1+\sum\limits_{y=\ell+1}^{\mu_{b}}\Floor{\frac{d-1-by}{a}}+ |H(d)|
\end{equation}
where
\begin{equation*}
|H(d)|=\begin{cases}\den(d;a,b) & \text{if } d \leq ab(q-1), \\
q-1+\mathbf{1}_{a \mid d} + \mathbf{1}_{b \mid d} & \text{if } d > ab(q-1).
\end{cases}
\end{equation*}
\end{theorem}
\begin{proof}
 By Theorem \ref{theo:basis_code}, the dimension of $C_d$ is equal to the number of lattice points inside $ \red(d)$. By Lemma \ref{lem:red}, we have 
 \[\red(d)=R(d)\cup T(d) \cup H(d).\]
 Clearly, the cardinality of the set $R(d)$ of lattice points of the rectangular area is
 \begin{equation}\label{eq:R(d)}
     \size{R(d)}=(\ell +1)(\mu_a+1)=(\ell +1)\mu_a + \ell + 1.
 \end{equation}
 As the definition of $\ell$ ensures that $\mu_b \geq \ell$, the set $T(d)$ is empty when $\mu_b= \ell$ so the sum in Equation \eqref{eq:dimab} means $0$. Furthermore, when $\mu_b \geq \ell+1$, $T(d)$ is the set of lattice points of a trapezoidal region with cardinality
 \begin{equation}\label{eq:T(d)}
\size{T(d)}=\sum\limits_{y=\ell+1}^{\mu_{b}}\left(\Floor{\frac{d-1-by}{a}}+1\right)=\mu_b-\ell+\sum\limits_{y=\ell+1}^{\mu_{b}}\Floor{\frac{d-1-by}{a}}.
 \end{equation}

It remains to deal with the cardinality of $H(d)=E_h(d)\cup [\{(0,\frac{d}{b}),(\frac{d}{a},0)\}\cap \Z^2]$, that is the set of lattice points of $\red(d)$ lying on the hypotenuse of the triangle $P_d$. It is clear that $\den(d;a,b)$ is the number of lattice points on the hypotenuse of the triangle $P_d$, which are equivalent, containing the corners if they have integral coordinates. As $|E_h(d)|=\min\set{q-1,i_{max}+1}$
, we have the following formula
\begin{equation}\label{eq:imax}
    i_{max}+1=\den(d,a,b) -\epsilon \text{ where }  \epsilon=\mathbf{1}_{a \mid d} + \mathbf{1}_{b \mid d}
\end{equation}
and $ \size{H(d)}$ can be reformulated as follows
\begin{equation}\label{eq:H(d)-rephrase}
\size{H(d)}=\min\set{q-1+\epsilon,\den(d;a,b)}.
\end{equation}

Recall that if $d=\lambda ab+s$ with $0\le s <ab$ then $\den(d;a,b)=\lambda +\den(s;a,b)$ with ${\den(s;a,b) \in \{0,1\}}$ depending on whether $s\in \la a,b \ra_\N$ or not.

\begin{itemize}
    \item When $d<ab(q-1)$, we have $\lambda <q-1$ and $\den(d;a,b)\leq q-1$. As $\epsilon \geq 0$, we clearly have $|H(d)|=\den(d;a,b)$.
    \item If $d=(q-1)ab$, then both $a|d$ and $b|d$ so $\epsilon=2$. Besides
    \[\den(d;a,b)= q-1 +\den(0;a,b)=q \leq q-1 + \epsilon.\]
    Therefore, $|H(d)|=\den(d;a,b)$.
    \item Now assume that $d> ab(q-1)$. Then $\den(d;a,b)\geq q-1 +\den(s;a,b)$.
    \begin{itemize}
        \item If $\epsilon=0$, Equation \eqref{eq:H(d)-rephrase} gives $\size{H(d)}=q-1$.
        \item If $\epsilon=1$, then either $a$ or $b$ divides $s$, which implies that $\den(s;a,b)=1$. So $\den(d;a,b) \geq q = q-1 + \epsilon$, hence $\size{H(d)}=q$.
        \item If $\epsilon=2$, then both $a$ or $b$ divides $d$, which means that $d=\lambda ab$ for some integer $\lambda \geq q$. Then, we have 
        \[\den(d;a,b)= \lambda +\den(0;a,b)=\lambda +1\geq q- 1 +\epsilon,\]
        and $|H(d)|=q+1$.
    \end{itemize}
\end{itemize}
The quantities in Equations \eqref{eq:R(d)}, \eqref{eq:T(d)} and \eqref{eq:H(d)-rephrase} add up to the number in the stated formula in Equation \eqref{eq:dimab}, completing the proof.
\end{proof}

\begin{corollary}\label{cor:11b} Let $Y=\P(1,1,b)(\F_q)$ for a positive integer $b$ and $d \geq 1$. Then, we have $\ell = \max\set{0,\min\set{q-1,\Floor{\frac{d-q}{b}}}}$ and $\dim_{\F_q}(C_d)$ is given by
\begin{equation*}\label{eq:dim1b}   
(\ell+1)(\mu_a+1)+ (\mu_b-\ell)d- b\binom{\mu_b+1}{2}+b\binom{\ell+1}{2}+ \mu_b+1+ \mathbf{1}_{b \mid d}.
\end{equation*}

In addition, we have $\ell=0$ and $\mu_a=d-1$ if $d\le q$. 
\end{corollary}

\begin{proof}
If $a=1$, then $\mu_a=\min\{q-1,d-1\}$ and Equation \eqref{eq:dimab} becomes
 \[\dim_{\F_q}(C_d) = 
(\ell+1)\mu_a+\mu_b+1+\sum\limits_{y=\ell+1}^{\mu_{b}}\left(d-1-by\right) + |H(d)|.\]

Rearranging this sum we get the following formula
\begin{align*}
 \dim_{\F_q}(C_d)&= (\ell+1)\mu_a+\mu_b+1+(\mu_b-\ell)(d-1)-b\sum\limits_{y=\ell+1}^{\mu_{b}} y+|H(d)|\\
 &=(\ell+1)(\mu_a+1)+(\mu_b-\ell)d-b\binom{\mu_b+1}{2}+b\binom{\ell+1}{2}+ |H(d)|.
\end{align*} 
Since $E_h(d)=\emptyset$ if $d\leq b$ and $N_0=(x',y')=(d-b,1)$ is the first interior lattice point if $d> b$, it follows from Equation \eqref{eq:Eh} that we have $$E_h(d)=|\{(d-by,y): 1\le y \le \mu_b\}|=\mu_b.$$ Therefore, we get $|H(d)|=\size{E_h(d)}+1+ \mathbf{1}_{b \mid d}=\mu_b+1+ \mathbf{1}_{b \mid d}$ as desired.
\end{proof}

\begin{remark}
The formula for the dimension agrees with \cite[Theorem 6.4]{GL2023} for the case $a=b=1$.
\end{remark}

\section{The Regularity Set}\label{sec:reg}

In this section, we describe the so--called regularity set of the $\F_q$-rational points $Y=\P(1,a,b)(\F_q)$ of the weighted projective plane $\P(1,a,b)$. The importance of this set stems from the fact that its elements give rise to trivial codes as we explain shortly. Recall that the Hilbert function of a \textit{zero dimensional} subvariety $Y$ of the weighted projective plane $\P(1,a,b)$ is defined to be the Hilbert function of the ring $S/I(Y)$:
\[
H_Y(d):=\dim_{\K} S_d - \dim_{\K} I_d(Y)=\dim_{\F_q} S_d - \dim_{\F_q} I_d(Y),
\]
where $I_d(Y):=I(Y)\cap S_d$ is the degree--$d$ part of the vanishing ideal $I(Y)$ generated by the homogeneous polynomials vanishing on $Y$. If $Y=\P(1,a,b)(\F_q)$, the kernel of the evaluation map in $\eqref{evmap}$ is $I_d(Y)$ and hence $\dim_{\F_q} C_d=H_Y(d)\le |Y|$.

\begin{definition}
The regularity set of $Y=\P(1,a,b)(\F_q)$ is defined by \[\reg (Y)=\{ d\in \N W : H_{Y}(d)=|Y|\}.\]
\end{definition}

\begin{remark}\label{rem:trivialcodes}
If $d\in \reg(Y)$ then the dimension of the code is maximum possible, i.e. $C_d=\F_q^n$, where $n=|Y|$. Thus, it is a trivial code with parameters $[n,n,1]$.
\end{remark}

\begin{proposition}\label{prop:reg}
Let $a$ and $b$ be two relatively prime integers with $1\leq a\leq b$. An integer $d$ belongs to the regularity set of $Y=\P(1,a,b)(\F_q)$ if and only if there exists $d_0 \geq q$ such that $d=d_0ab$ and $(a+b)(q-1) < d$.
\end{proposition}

\begin{proof}
    It follows from the description of $\red(d)$ given by Lemma \ref{lem:red} that $d\in \reg(Y)$ if and only if $|\red(d)|=|Y|=q^2+q+1=(q-1)^2+3(q-1)+3$ if and only if $|P_d^{\circ}|=(q-1)^2$, $|E_x(d)|=|E_y(d)|=|E_h(d)|=q-1$,  $a \mid d$ and $b\mid d$.

    One can easily check that $|P_d^{\circ}|=(q-1)^2$ if and only if the point $(q-1,q-1)$ lies in the interior of $P_d$, which is equivalent to $(a+b)(q-1)<d$. In this case, $(q-1,0) \in E_x(d)$ and $(0,q-1)\in E_y(d)$, so $|E_x(d)|=|E_y(d)|=q-1$.

    Since the integers $a$ and $b$ are coprime, $a \mid d$ and $b\mid d$ if and only if $ab \mid d$ which is equivalent to the fact that there exists an integer $d_0$ such that $d=d_0ab$. 
    Then $\den(d;a,b)= d_0+1$. By Equation \ref{eq:imax},
    \[|E_h(d)|=\min\set{q-1,\den(d;a,b)-2} =q-1 \: \Leftrightarrow \: d_0 \geq q, \]
    which concludes the proof.
\end{proof}

\begin{theorem}\label{theo:reg1ab}
    Let $a$ and $b$ be two relatively prime positive integers with $1<a<b$. The regularity set of $Y=\P(1,a,b)(\F_q)$ is given as follows. 
\[\reg(Y)=\{d \in \N: d=d_0ab \mbox{ with } d_0\ge q\}=qab+\N ab.\]
\end{theorem}

\begin{proof}
    By Proposition \ref{prop:reg}, it is enough to check that
    writing $d=d_0ab$ with $d_0 \geq q$ implies that $(a+b)(q-1) < d$.
        
    Since $a\ge 2$ and $b>a$ we have $ab\ge 2b >a+b$, leading to the following 
    \[d=d_{0}ab\ge qab \ge q(a+b)>(q-1)(a+b)\]
completing the proof.
\end{proof}

We also recover \cite[Corollary 3.9]{CS2024} more directly in the case $a=1$.

\begin{theorem}\label{theo:reg11b}
    Let $b$ be a positive integer. The regularity set of $Y=\P(1,1,b)(\F_q)$ is given as follows. 
\[\reg(Y)=\{d \in \N: d=d_0b \mbox{ with } d_0\ge q+\lfloor (q-1)/b\rfloor\}=\left(q+\Floorfrac{q-1}{b}\right)b+\N b.\]
\end{theorem}

\begin{proof} By noticing that 
        \[d_0b > (1+b)(q-1) \: \Longleftrightarrow \: d_0 \geq \Floorfrac{q-1}{b}  + q,\]
the proof follows straightforwardly from Proposition \ref{prop:reg}.
\end{proof}

\section{A lower bound for the minimum distance}\label{sec:min-dist}

In this section we aim at grasping the minimum distance of a projective weighted Reed-Muller code obtained from $\Pp(1,a,b)$ by first establishing a lower bound.
 For every $f\in S_d\setminus I_d(Y)$, we denote by
$Y_f$ the subvariety $V_Y(f)$ of $Y$ consisting of the roots of $f$ and we write $n_f:= \size{Y_f}$. As usual for evaluation codes, we have
\[d_{min}(C_d)= n-\max \{ n_f : f\in S_d\setminus I_d(Y) \}.\]

By the virtue of Remark \ref{rem:trivialcodes}, from now on, we assume that $d\in \N \setminus \reg(Y)$. To give a lower bound on the minimum distance, it suffices to give an upper bound on $n_f$. Such a bound can be obtained using Gr\"obner basis theory, which is known in the literature  as \textit{the footprint bound}, see e.g. \cite{BDG2019,GH2000,Villarreal2017} for the affine and projective space cases, \cite{JNHir19} for Hirzebruch surfaces, and \cite{JNPro22} for more general toric varieties.

The \textit{footprint} of a homogeneous ideal $I$ of $S$ with respect to a monomial ordering $\prec$ on the set $\M$ of all monomials in $S$ is defined as follows
\[\Delta(I):=\{M\in \M: M \neq \lm(g) \text{ for any } g\in I \text{ with } g\neq 0\}.\] 
A relevant \textit{finite} subset is $\Delta_{\Tilde{d}}(I)=\Delta(I) \cap \M_{\Tid}$ for $\Tid \in \N$. It is well known from Gr\"obner basis theory that a monomial $M \neq \lm(g) \text{ for any } g\in I\setminus \{0\}$ if and only if $\lm(g) \nmid M$ for any $g$ in the Gr\"obner basis of $I$ with respect to $\prec$.
By \cite[Proposition 1.1]{S96}, the value $H_I(\Tid)$ of the Hilbert function of $I$ at $\Tid \in \N$ is the cardinality of $\Delta_{\Tilde{d}}(I)$. In particular, it implies that $\size{\Delta_{\Tilde{d}}(I)}$ is independent of the ordering $\prec$.

\bigskip

We rewrite \cite[Lemma 5.5]{JNPro22} in the notation used here.

\begin{lemma}\cite[Lemma 5.5]{JNPro22}\label{lem: footprintbound} Let $Y=\P(1,a,b)(\F_q)$ and $I(Y,f)$ denote the ideal $I(Y)+(f)$, for a homogeneous polynomial $f\in S$. Then for any $f\in S_d \setminus I_d(Y)$ and $\tilde{d} \in \reg(Y)$, we have $n_{f} \leq \Tilde{n}_f(\tilde{d} ):=H_{I(Y,f)}(\tilde{d} )=|\Delta_{\Tilde{d}}(I(Y,f))|.$ \qed
    \end{lemma}
    
By \cite[Theorem 4.3.5]{BJ1993} there exists a uniquely determined quasi-polynomial $P_{M}$ with $H_{M}(d)=P_{M}(d)$ for all $d>a(M)$ for the module $M=S/I(Y,f)$, where $a(M)$ denotes the degree of the rational function representing the Hilbert series $HS_{M}(t)$. In other words, there exists a positive integer $g$ (period) and some polynomials $P_{0},\dots,P_{g-1}$ such that $H_{M}(d)=P_{i}(d)$ for $d>a(M)$ and $d\equiv i\mod g.$

Since $I(Y)\subseteq I(Y,f) \subseteq I(Y_f)$, the Krull dimension of $I(Y,f)$ is $1$, if $f$ is a zerodivisor of $S/I(Y)$. Therefore, the Hilbert quasi-polynomial of $S/J$ has degree $0$ if $J=I(Y,f)$ and $f$ is a zerodivisor of $S/I(Y)$, see \cite[Proposition 5]{CM15} or \cite{bavula}. Therefore, there are finitely many constants appearing as the values of $\Tilde{n}_f(\tilde{d} )$ for sufficiently large $\Tid$. We thus define $\Tilde{n}_f$ as the maximum of all these constants. This gives the following lower bound for the minimum distance.

\begin{equation}\label{eq:bound_algebra}
d_{min}(C_d)\geq n-\max \set{ \Tilde{n}_f : f\in S_d\setminus I_d(Y) \text{ is a zerodivisor of } S/I(Y) }.
\end{equation}

\subsection{A lower bound for the minimum distance using Gr\"obner basis theory}

In practice, it is difficult to compute Hilbert quasi-polynomials to get $\Tilde{n}_f$. Therefore, we appeal to a classical trick (see \cite[Lemma 2.5]{Villarreal2017} for instance) of Gr\"obner basis theory to give an easy upper bound for $\Tilde{n}_f(\Tid)$, which will be independent of $\Tid$.

\begin{lemma} \label{lem:U}
Let $X$ be a simplicial toric variety and $Y \subset X$ a subvariety.

Let $\cG$ denote the Gr\"obner basis of $I(Y)$ with respect to a term order $\prec$. Let 
\[\overline{\Delta}_{\Tid}(f)= \set{ M \in \M_{\Tid} : \forall g \in \cG \cup \set{f},\: \lm(g) \nmid M}. \]

Then, for any $f\in S_d \setminus I_d(Y)$ we have
\begin{equation} \label{eq:deltaf}
 \Tilde{n}_f(\Tid)=H_{I(Y,f)}(\Tid) \leq \size{\overline{\Delta}_{\Tid}(f)},   
\end{equation}

and $\Tilde{n}_f(\Tid)= \size{\overline{\Delta}_{\Tid}(f)}\iff \cG \cup \set{f}$ is the Gr\"obner basis of $I(Y,f)$ for $\prec$. \qed
\end{lemma}

In our case, \textit{i.e.} $X=\Pp(1,a,b)$ and $Y=X(\F_q)$, the set of monomials of degree $\Tid$ that are not divisible by $\lm(g)$ for all $g\in \cG $ (for $\prec \:= \: <_{\text{lex}}$)  is $\overline{\M}_{\Tid}$. Thus,
\begin{equation}\label{eq:Delta_f}
    \overline{\Delta}_{\Tid}(f)= \set{ M \in \overline{\M}_{\Tid} : \lm(f) \nmid M}.
\end{equation}

For the upper bound $\size{\overline{\Delta}_{\Tid}(f)}$, it suffices to consider $f$ to be a monomial as $\overline{\Delta}_{\Tid}(f)=\overline{\Delta}_{\Tid}(\lm(f))$.
We can also assume that $\lm(f)=\x^{\a,d}$ with $\a \in \red(d)$. Thus, for a fixed $\a \in \red(d)$ and $\Tid \geq d$, we consider the \emph{projective shadow} in degree $\Tid$ of the monomial $\x^{\a,d}$ 
\[\overline{\nabla}_{\Tid}(\x^{\a,d}) = \set{ \x^{\tilde{\a},\Tid} \in \overline{\M}_{\Tid} : \x^{\tilde{\a},\Tid} \text{ is divisible by } \x^{\a,d}}= \overline{\M}_{\Tid} \setminus \overline{\Delta}_{\Tid}(\x^{\a,d}). \]

\begin{corollary}\label{cor:U_lex}
Consider the lexicographic order $\prec \:= \: <_{\text{lex}}$. Suppose that $\Tid\in \reg(Y)$. Then for every $f \in S_d \setminus I(Y)$, we have
\[
\size{\overline{\Delta}_{\Tid}(f)}=\size{\overline{\Delta}_{\Tid}(\lm(f))}=n-\overline{\nabla}_{\Tid}(\lm(f)).
\]
\end{corollary}
\begin{proof}

When $\Tid\in \reg(Y)$, $|\overline{\M}_{\Tid}|=H_{I(Y)}(\Tid) =n$. The formula for the cardinality of $\size{\overline{\Delta}_{\Tid}(f)}$ follows directly from Equation \eqref{eq:Delta_f}.
\end{proof}

For every $\a \in \red(d)$ and $\Tid \in \reg(Y)$, we set
\begin{equation}\label{eq:defL}
L(\a,\Tid):=\size{\overline{\nabla}_{\Tid}(\x^{\a,d})}=n-\size{\overline{\Delta}_{\Tid}(\x^{\a,d})}.
\end{equation}

\begin{lemma} \label{lem:mindist}
Let $d\in \N \setminus \reg(Y)$ and $\Tid \in \reg(Y)$. Then
\[
d_{min}(C_d)\ge \min \set{L(\a,\Tid): \a \in \red(d)}.
\]
\end{lemma}
\begin{proof} Clearly, we have 
\[
 d_{min}(C_d)= n-\max \{|V_Y(f)|: f\in S_d\setminus I_d(Y)\}.
 \]
 Let $f\in S_d\setminus I_d(Y)$ and write $\lm(f)=\x^{\a,d}$. By Theorem \ref{theo:basis_code}, we can assume that $\a \in\red(d)$. Lemmas \ref{lem: footprintbound} and \ref{lem:U} imply that $|V_Y(f)|\le \size{\overline{\Delta}_{\Tid}(\lm(f))}$, and thus
 \[
 d_{min}(C_d)\ge n-\max \{\size{\overline{\Delta}_{\Tid}(\x^{\a,d})}: \a\in \red(d)\} = \min \{L(\a,\Tid): \a \in \red(d)\},
 \] 
as claimed.
\end{proof}

\begin{remark} \label{rem:L>0} Let $\a \in \red(d)$.
When $\Tid \in \reg(Y)$, the monomials $x_0^{\Tid}$, $x_1^{\Tid/a}$ and $x_2^{\Tid/b}$ must belong to $\overline{\M}_{\Tid}$. So, if $\x^{\a,d}$ divides two of these, it must be a constant forcing $d=0$. Then if $d \geq 1$, we have $L(\a,\Tid) \le n-2$.
Moreover, if $d \leq \Tid$, the monomial $\x^{\a,d}$ divides $\x^{\a,\Tid}$, resulting in $L(\a,\Tid)\geq 1$ as $d\notin \reg(Y)$. As a consequence, the bound provided by Lemma \ref{lem:mindist} is not trivial.
\end{remark}

\subsection{Computation of the lower bound for the minimum distance.}

Our strategy is the following. We first get rid of the dependency of the lower bound on the degree $\Tid$: Proposition \ref{prop:func_L} provides a formula for the cardinality of the projective shadow of monomials in $\overline{\M}_d$ in degree $\Tid$ that do not depend on $\Tid$. 

We will then determine the minimum of $L=L(\cdot,\Tid)$ on the convex area $\red(d) \cap \set{a_2 \leq \mu_b}$ in \textsection \ref{subsubsec:minleq_q-1}.
One can easily check from the definition of $\red(d)$ that
\begin{itemize}
    \item if $d \leq bq-1$, then $\mu_b=\Floorfrac{d-1}{b}$ and $a_2 \leq \mu_b$ for every $\a \in \red(d)$ except for $\a=(0,d/b)$ in the case $b \mid d$ (see Figure \ref{fig:case1} for an example);
    \item if $d \geq bq$, then $\mu_b=q-1$. If $d-bq \in \langle a,b\rangle$, then there exists some $\a=(a_1,a_2) \in \red(d)$ such that $a_2 \geq q$. By definition of $\red(d)$ such an element $\a$ should belong to $H(d)$ (see Figure \ref{fig:case2} for an example). In this case, we will have to compare the minimum on the area $\red(d) \cap \set{a_2 \leq \mu_b}$ and the values outside, which is the heart of \textsection \ref{sec:lowerbound}.
\end{itemize}

\begin{figure}[h]
\begin{subfigure}[b]{0.44\textwidth}
\centering
\includegraphics[width=\textwidth]{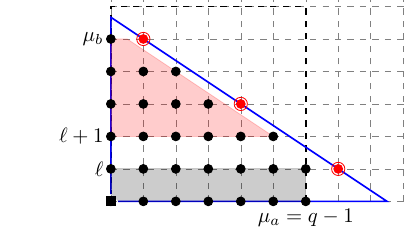}
\caption{$d\leq bq-1$ then $\mu_b=\Floor{\frac{d-1}{b}}$ and \\$a_2\leq\mu_b$ for every $\a\in \red(d)$} \label{fig:case1}
\end{subfigure}
\begin{subfigure}[b]{0.54\textwidth}
\centering
\includegraphics[width=\textwidth]{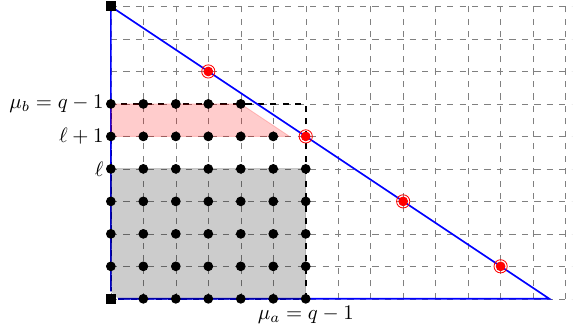}
\caption{$d\ge bq$, then $\mu_b=q-1$.}\label{fig:case2}
\label{fig:case2}
\end{subfigure}
\caption{Lattice points $\a\in\red(d)$ for $\P(1,2,3)(\F_7)$, and $d=17$ and $d=27$ respectively.}
\end{figure}

\begin{proposition} \label{prop:func_L}
For $\Tid \in \reg(Y)$ satisfying $\Tid \geq d + (q-1)\max\{a+b,ab\}$ and for every $\a=(a_1,a_2)\in\red(d)$, the quantity $L(\a,\Tid)$ does not depend on $\Tid$ and equals the following number denoted by $L(\a)$
\begin{equation*}\label{eq:func_L}
    \begin{cases}
        (q-a_1)(q-a_2) & \text{if } aa_1+ba_2\neq d\\
         \max\{q-a_1,0\}\max\{q-a_2,0\} + q- 1 - \Floor{\frac{a_2-1}{a}} & \text{if } aa_1+ba_2=d \text{ and } a_1 \neq 0,\\
        q\cdot\max\{q-\frac{d}{b},0\} + \max\{q- \Floor{\frac{d-b}{ab}},1\} & \text{if } (a_1,a_2)=(0,\frac{d}{b}) \text{ and } b \mid d.   
    \end{cases}
\end{equation*}
\end{proposition}

\begin{proof}
Fix $\a=(a_1,a_2)\in\red(d)$ and let $\x^{\a,d}=x_0^{d-aa_1-ba_2}x_1^{a_1}x_2^{a_2}$. Let us count the number of couples $\tilde{\a}=(\Tia_1,\Tia_2)\in\red(\Tid)$ such that the monomial $\x^{\a,d}$ divides $\x^{\tilde{\a},\Tid}=x_0^{\Tid-a\Tia_1-b\Tia_2}x_1^{\Tia_1}x_2^{\Tia_2}$. We clearly have the following equivalence.
\begin{equation}\label{eq:divisibility}
 \x^{\a,d} \text{ divides } \x^{\tilde{\a},\Tid} \iff \left\{\begin{array}{rcl}
     a_1 &\leq& \Tia_1,\\
     a_2 &\leq& \Tia_2 \\
     d-aa_1-ba_2 &\leq& \Tid-a\Tia_1-b\Tia_2.
 \end{array}\right.   
\end{equation}

Before going further, let us notice that the hypothesis $\Tid \in \reg(Y)$ ensures that $|\red(\Tid)|=q^2+q+1$. In particular, $\red(\Tid)=R(\Tid) \cup H(\Tid)$ with
\begin{equation}\label{eq:Ptid}
R(\Tid)=\{(\Tia_1,\Tia_2)\in \Z^2 : 1\leq \Tia_1 \leq q-1 \text{ and } 1\leq \Tia_2 \leq q-1\}
\end{equation}
and 
\begin{equation}\label{eq:Htid}
H(\Tid)=\{(b(d_0-y_0),y_0a) \in \Z^2: 0\leq y_0 \leq q-1 \text{ and }y_0=d_0\},
\end{equation}
where $\Tid=d_0ab$, for some $d_0\in \N$ with $d_0\geq q$ (see Proposition \ref{prop:reg}).

\medskip

If $\a=(a_1,a_2)\in \red(d)\setminus H(d)$ then $aa_1+ba_2< d$ and so a necessary requirement for the third condition to hold in \eqref{eq:divisibility} is to have $a\Tia_1+b\Tia_2 < \Tid$ as well. This means that $(\Tia_1,\Tia_2)\in R(\Tid)$ and so $\Tia_1 \leq q-1$ and $\Tia_2 \leq q-1$. Then
\[\Tid -d \geq (q-1)(a+b)\geq a\Tia_1+b\Tia_2 \geq a(\Tia_1-a_1)+b(\Tia_2-a_2),\]
implying the third condition  in \eqref{eq:divisibility}.
Hence, all the conditions in \eqref{eq:divisibility} hold for $0 \leq a_1 \leq \Tia_1 \leq q-1$ and $0 \leq a_2 \leq \Tia_2 \leq q-1$. Therefore, $\Tia_1$ and $\Tia_2$ respectively admit $q-a_1$ and $q-a_2$ possible values, proving the first description of $L$. 

\medskip

Now assume that $\a\in H(d)\setminus \{(0,d/b)\}$. Then, $aa_1+ba_2= d$, which means that the third condition in $(\ref{eq:divisibility})$ is satisfied, and we have $\x^{\a,d}$ divides $\x^{\tilde{\a},\Tid}$ if and only if $a_1 \leq \Tia_1$ and $a_2 \leq \Tia_2$. The number of such tuples $\tilde{\a} \in R(\Tid)$ is given by $\max\{q-a_1,0\}\max\{q-a_2,0\}$, as $a_1$ and $a_2$ may be larger than $q-1$ but $\Tia_1 \leq q-1$ and $\Tia_2 \leq q-1$. Note that $\a\neq(0,d/b)$ implies that $a_1 \geq 1$ and $0 \leq a_2 < aq$, see Equation \eqref{eq:Eh}. By Equation \eqref{eq:Htid}, for every $\tilde{\a}=(\Tia_1,\Tia_2)\in H(\Tid)$ we have
\[\begin{aligned}
  a_1 \leq \Tia_1 &\Longleftrightarrow& d-ba_2 \leq \Tid -aby_0 \\
                  &\Longleftrightarrow& aby_0\leq \Tid - d +ba_2
\end{aligned}\]
so the hypothesis $\Tid \geq d + (q-1)\max\{a+b,ab\}$ reduces the condition $a_1 \leq \Tia_1$ to $\Tia_1 \geq 1$ for every $(\Tia_1,\Tia_2) \in H(\Tid)\setminus \{(0,\Tid/b)\}$. As there are exactly $\Floorfrac{a_2-1}{a}$ tuples $\tilde{\a}$ such that $\Tia_2= ay_0 < a_2$,  we obtain the formula
\[L(\a)=\max\{q-a_1,0\}\max\{q-a_2,0\}+q-1-\Floorfrac{a_2-1}{a}\]
proving the second case.

%
%
%
%


\medskip

Finally, assume that $\a=(0,d/b)$ with $b\mid d$. The only difference with the previous case is that $\Tia_1$ can also be $0$. So, inserting $a_1=0$, $a_2=d/b$ in the previous formula, and adding $1$ for the corner $(0,\Tid/b)$, we obtain the last formula:
$q\cdot\max\{q-\frac{d}{b},0\} + \max\{q- \Floor{\frac{d-b}{ab}},1\}$, completing the proof.
\end{proof}
\begin{remark}
Even though the second formula for $L(\a)$ in Proposition \ref{prop:func_L} can take on the value $0$, this occurs only if $\a \notin \red(d)$, by the virtue of Remark \ref{rem:L>0}. 

We recover \cite[Lemma 4.1]{BDG2019} when plugging $a=b=1$ in Proposition \ref{prop:func_L}.
\end{remark}

The next example shows that the lower bound for $\Tid$ in Proposition \ref{prop:func_L} is sharp.

\begin{example} \label{ex:example1}
 Let $q=2$, $X=\P(1,1,3)$ and $Y=X(\F_q)$. By Theorem \ref{thm:UniversalBasis}, the following is the universal Gr\"obner basis for $I(Y)$: 
$$\mathcal{G}=\{ f_0=x_2^{2}x_1+{x}_{2}x_1^{4},f_1=x_2^{2}x_0+x_2x_0^{4},  f_2=x_1^{2}x_0+x_1x_0^{2}\}.$$ 

One can compute the minimum distance of the code $C_{4,Y}$ to be $2$ using one of the algorithms given by \cite{BalSah23}. This reveals that a homogeneous polynomial $f$ having the maximum possible number of roots has $n_f=|Y|-2=5$. We use this to analyze the smallest possible element $\Tid$ for which Proposition \ref{prop:func_L} works.

 For $
d=(q-1)(a+b)=4$, a basis for $S_{d}/I_{d}(Y)$ is given by 
$$\overline{\M}_4=\{x_2x_1,\,x_2x_0,\,x_1^{4},\,x_1x_0^{3},\,x_0^{4}\}.$$
For any $\Tid \in \reg(Y)$, we have
$$\overline{\M}_{\Tid}=\{x_2^{\Tid/3}, x_2x_1^{\Tid-3}, x_2x_1x_0^{\Tid-4}, x_2x_0^{\Tid-3}, x_1^{\Tid}, x_1x_0^{\Tid-1},  x_0^{\Tid}\}.$$
Using the formulae of Proposition \ref{prop:func_L}, we get
\[L(a_1,a_2)=\begin{cases} 4 & \text{if } (a_1,a_2)=(0,0), \\
2 & \text{if } (a_1,a_2)\in\set{(1,1),(0,1),(4,0),(1,0)}.
\end{cases} \]

Since $\reg(Y)=6+3\N$ by Theorem \ref{theo:reg11b}, we first look at the values of the Hilbert function of $I(Y)+\la f \ra$ for every $f\in \overline{\M}_4$ and get
\[ H_{I(Y)+\la f \ra}(6)=\begin{cases}
6 & \text{for } f=x_1x_0^{3},\\
5 & \text{for }  f\in \overline{\M}_4\setminus \{x_1x_0^{3}\}.
\end{cases}
\]
This reveals that for $f=x_1x_0^{3}$, $H_{I(Y)+\la f \ra}(6)$ is bigger than the biggest $n_f$.

However, taking $\Tid \geq 2(q-1)(a+b)=8$, \textit{i.e.} $\Tid \in 9+3\N$, remedies the problem. Indeed, one can observe that $\Tilde{n}_f(\Tid)=\size{\overline{\Delta}_{\Tid}(f)}=5$ for all $\Tid \in 9+3\N$ and for $f=x_1x_0^{3}$, as the Gr\"obner basis of $I(Y)+\la f \ra$ is $\mathcal{G}\cup \{f\}$ and a basis for $S_{\Tid}/I_{\Tid}(Y,f)$ is obtained as $\overline{\Delta}_{\Tid}(f)=\{x_2^{\Tid/3}, x_2x_1^{\Tid-3}, x_2x_0^{\Tid-3}, x_1^{\Tid},   x_0^{\Tid}\}.$  

For any $\Tid \in 9+3\N$, we check that $\Tilde{n}_f(\Tid)=\size{\overline{\Delta}_{\Tid}(f)}=5$  for $f\in \{x_2x_1,\,x_2x_0\}$, $\Tilde{n}_f(\Tid)=3<\size{\overline{\Delta}_{\Tid}(f)}=5$  for $f=x_1^{4}$ and $\Tilde{n}_f(\Tid)=\size{\overline{\Delta}_{\Tid}(f)}=3$  for $f=x_0^{4}$. 
\end{example}

Example \ref{ex:example1} shows that the minimum of the function $L$ on the set $\red(d)$ provided by Lemma \ref{lem:mindist} seems to give the exact minimum distance. Let us start with large degrees $d$, corresponding to high rate codes $C_d$, for which the minimum value of the function $L$ is exactly equal to 1.

\begin{proposition}\label{prop:trivialcase}
If $\Tid \geq d > (a+b)(q-1)$, then $\min_{\a \in \red(d)} L(\a,\Tid)=1$.
\end{proposition}

\begin{proof}
The hypothesis $d > (a+b)(q-1)$ is equivalent to $(q-1,q-1)$ lying in $P_d^\circ$, hence in $\red(d)$. The minimum of $L$ is thus clearly reached at this point, since $\x^{\tilde{\a},\Tid}$ is divisible by $\x^{\a,d}=x_0^{d-(q-1)(a+b)} x_1^{q-1} x_2^{q-1}$ if and only if $\Tia_1=q-1$, $\Tia_2=q-1$ and $\Tid \geq d$.
\end{proof}

\subsubsection{Minimum of $L(a_1,a_2)$ on $\red(d)\cap \set{a_2 \leq \mu_b}$}\label{subsubsec:minleq_q-1}

Our first observation eliminates lots of the points $(a_1,a_2)\in \red(d)\cap \set{a_2 \leq \mu_b}$ for which the function $L(a_1,a_2)$ does not attain its minimum value. In other words, it does not only say that the minimum is reached at an interior point but it also describes the $a_1$ coordinate of such a point on each horizontal line $y=a_2$ (see Figure \ref{fig:argmin}).

\begin{figure}[h]
    \centering
    \includegraphics[scale=0.80]{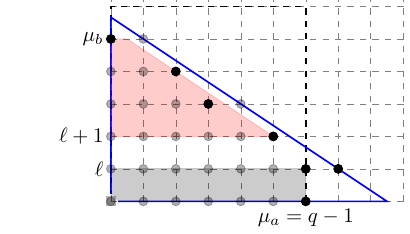}
    \includegraphics[scale=0.70]{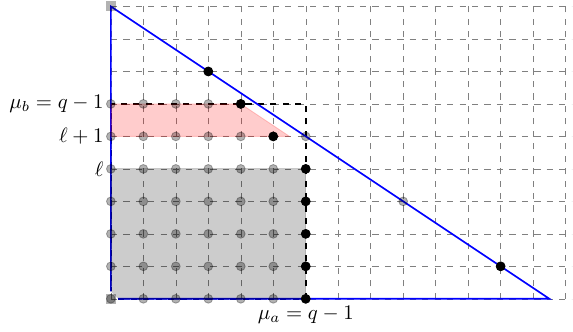}
    \caption{Candidates for the minimum of $L$ on \mbox{$\red(d) \cap \set{a_2 \leq \mu_b}$} by virtue of Lemma \ref{lem:argmin} with $a,b=2,3$, $q=7$ and the degrees $d=17,27$, respectively.}
    \label{fig:argmin}
\end{figure}

\begin{lemma}\label{lem:argmin}
Fix $a_2 \in \set{0,\dots, \mu_b}$. Write $\calX_{a_2}=\set{a_1 \in \N: (a_1,a_2) \in \red(d)}$ and $M_{a_2}=\max \calX_{a_2}$. 

If $aM_{a_2}+ba_2< d$, then the minimum of the univariate function $a_1 \mapsto L(a_1,a_2)$ on $\calX_{a_2}$ is reached exactly at $a_1=M_{a_2}$.

If $aM_{a_2}+ba_2= d$, then it is reached at
\[\underset{a_1 \in \calX_{a_2}}{\argmin} \: L(\cdot,a_2) =
\begin{cases}
\set{\min\set{M_{a_2},q}-1,M_{a_2}} & \text{if } a_2=0,1 \text{ or } a=1,\\
\set{\min\set{M_{a_2},q}-1}     & otherwise.
\end{cases}
\]
\end{lemma}

\begin{proof}
If $aM_{a_2}+ba_2\neq d$, then $(M_{a_2},a_2) \notin H(d)$ which means that $M_{a_2} \leq q-1$. Therefore the function $a_1 \mapsto L(a_1,a_2)$ is defined by $L(a_1,a_2)=(q-a_1)(q-a_2)$ on the set $\calX_{a_2}$ by Proposition \ref{prop:func_L}. It is thus strictly decreasing with respect to $a_1$.

\medskip

Otherwise, $aM_{a_2}+ba_2= d$. Then $M_{a_2} \geq 1$ (as $a_2 \leq \mu_b$) and it is clear that the function $a_1 \mapsto L(a_1,a_2)$ is strictly decreasing on $\calX_{a_2} \setminus \set{M_{a_2}}$ and we have two cases:
\begin{enumerate}
\item $M_{a_2} \leq q$, which implies that $\calX_{a_2}=\set{0,\dots,M_{a_2}}$.
\item $M_{a_2} > q$, which corresponds to $\calX_{a_2}=\set{0,\dots,q-1}\cup\set{M_{a_2}}$.
\end{enumerate}

In the first case, it remains to compare 
\[L(M_{a_2},a_2) = (q-M_{a_2})(q-a_2) +q - 1 - \Floor{\frac{a_2-1}{a}} \]
and 
$L(M_{a_2}-1,a_2)=(q-M_{a_2}+1)(q-a_2)=(q-M_{a_2})(q-a_2)+q-a_2$. Computing the difference, we get
\[\begin{aligned}
L(M_{a_2},a_2)-L(M_{a_2}-1,a_2) &= -1 - \Floor{\frac{a_2-1}{a}}+a_2\\
                        &= \Ceil{\frac{(a-1)(a_2-1)}{a}}\geq 0
\end{aligned}\]

In the second case, we just have to compare $L(M_{a_2},a_2)=q- 1 - \Floor{\frac{a_2-1}{a}}$ and $L(q-1,a_2)=q-a_2$. Their difference also equals to
\[
L(M_{a_2},a_2)-L(q-1,a_2) = \Ceil{\frac{(a-1)(a_2-1)}{a}}.
\]

In both cases, the minimum of the function $a_1 \mapsto L(a_1,a_2)$ on $\calX_{a_2}$ is thus reached by $\min\set{M_{a_2},q}-1$, and also by $M_{a_2}$ if $a_2 \in \{0,1\}$ or $a=1$.
\end{proof}

Let us define the following univariate function:
\begin{equation}\label{eq:tildeL}
a_2 \mapsto \tilde{L}(a_2)=L\left(\min\set{\Floor{\frac{d-1-ba_2}{a}},q-1},a_2\right)
\end{equation}

\begin{corollary}\label{cor:min_L_and_tildeL}Under the hypothesis above, we have
\[\min_{\substack{(a_1,a_2) \in \red(d)\\ a_2 \leq \mu_b}} L(a_1,a_2)= \min_{a_2 \in \set{0,\dots,\mu_b}} \tilde{L}(a_2) \] 
\end{corollary}
\begin{proof}
This directly follows from Lemma \ref{lem:argmin}, noticing that 
\[
\min\set{\Floor{\frac{d-1-ba_2}{a}},q-1}=\begin{cases}
    M_{a_2} & \text{if } aM_{a_2}+ba_2< d, \\
    \min\set{M_{a_2},q}-1 &  \text{if } aM_{a_2}+ba_2= d,
\end{cases}
\]
completing the proof. \end{proof}

Recall the number $\alpha_2=\Floor{\frac{d-1-a(q-1)}{b}}$ as described in Remark \ref{rem:cases}. To obtain the minimum value of the function $L(a_1,a_2)$ on $\red(d)\cap \set{a_2 \leq \mu_b}$, let us study how the univariate function $\tilde{L}$ varies on $\set{0,\dots,\mu_b}$.

\begin{remark}\label{rem:symmetry}
The hypothesis $a \leq b$ implies that if $(a_1,a_2) \in \red(d) \setminus H(d)$ with $a_2 > a_1$, then so does $(a_2,a_1)$. Moreover, the function $L$ defined in Proposition \ref{prop:func_L} is clearly symmetric on the interior of $P_d$, \textit{i.e.} $L(a_1,a_2)=L(a_2,a_1)$ if both $(a_1,a_2)$ and $(a_2,a_1)$ belong to $\red(d) \setminus H(d)$. Therefore, when investigating for the minimum of $L$ on $\red(d) \setminus H(d)$, we can restrict to the subset where $a_2 \leq a_1$.  
\end{remark}

By Remark \ref{rem:symmetry}, it is enough to study $\tilde{L}$ on $\set{0,\dots,\Floor{\frac{d-1}{a+b}}}$. 
Let us make the form of $\tilde{L}$ explicit on $\set{0,\dots,\Floor{\frac{d-1}{a+b}}}$.

\begin{lemma}\label{lem:formula_tildeL}

On the set $\set{0,\dots,\Floor{\frac{d-1}{a+b}}}$, the function $\tilde{L}$ defined in \eqref{eq:tildeL} is given as follows.
    \begin{enumerate}
    \item\label{it:formula_tildeL1} If $d \leq a(q-1)$ (\textit{i.e.} $\alpha_2 < 0$), then 
\[
\tilde{L}(a_2)=  \left(q-\Floor{\frac{d-1-ba_2}{a}}\right)(q-a_2).
\]
    \item\label{it:formula_tildeL2} If $d > a(q-1)$ (\textit{i.e.} $\alpha_2 \geq 0$), then 
\[
\tilde{L}(a_2)=\begin{cases}
    q-a_2 & \text{if } a_2 \leq \alpha_2, \\
    \left(q-\Floor{\frac{d-1-ba_2}{a}}\right)(q-a_2) &\text{otherwise.}
\end{cases}
\]
\end{enumerate}
\end{lemma}

\begin{proof} The first description follows from Proposition \ref{prop:func_L} and Equation \eqref{eq:tildeL} by the virtue of the observation that $ \Floor{\frac{d-1-ba_2}{a}}\leq  \Floor{\frac{d}{a}}\leq q-1$. The second part follows similarly, by noticing that

\begin{equation}\label{eq:cdt_a2_tildeL}
    \Floor{\frac{d-1-ba_2}{a}}\geq q-1 \: \Longleftrightarrow \: a_2 \leq \Floor{\frac{d-1-a(q-1)}{b}}=\alpha_2,
\end{equation}
finishing the proof.
\end{proof}

\begin{lemma}\label{lem:tildeL_variation}
If $\tilde{L}(a_2)= \left(q-\Floor{\frac{d-1-ba_2}{a}}\right)(q-a_2)$ as in Lemma \ref{lem:formula_tildeL}, then it is strictly increasing on the set $\set{\max\set{0,\alpha_2},\dots,\Floor{\frac{d-1}{a+b}}}$. 
\end{lemma}

\begin{proof}
For all the values of $a_2$ satisfying 
\[\max\set{0,\Floor{\frac{d-1-a(q-1)}{b}}} + 1 \leq a_2 \leq \min\set{\Floor{\frac{d-1}{a+b}},q-1},
\] 
the difference $\tilde{L}(a_2) - \tilde{L}(a_2-1)$ is
\[\begin{aligned}
  (q-a_2+1)\left( \Floor{\frac{d-1-b(a_2-1)}{a}} - \Floor{\frac{d-1-ba_2}{a}} \right) - q + \Floor{\frac{d-1-ba_2}{a}}.
\end{aligned}\]
If $a=1$, the expression above boils down to
\[\tilde{L}(a_2) - \tilde{L}(a_2-1)=d+(q+1)(b-1)-2a_2b\]
which is positive because $a_2 \leq \Floorfrac{d-1}{b+1}$ and $a_2 \leq q-1$.

If $a \neq 1$, note that  for all $ x,y \in \Z$, we have
\begin{equation}\label{eq:diff-floor}
 \Floor{\frac{x}{a}}-\Floor{\frac{y}{a}}=\Floor{\frac{x}{a}} - \Ceil{\frac{y+1}{a}}+1 \geq \Floor{\frac{x-y-1}{a}}.
\end{equation}
Applying this to $x=d-1-b(a_2-1)$ and $y=d-1-ba_2$, we get 
\[\tilde{L}(a_2) - \tilde{L}(a_2-1) \geq (q-a_2+1)\Floor{\frac{b-1}{a}} - q + \Floor{\frac{d-1-ba_2}{a}}.\]
Since we have $a_2 \leq \Floor{\frac{d-1}{a+b}}$, which is tantamount to $\Floor{\frac{d-1-ba_2}{a}} \geq a_2$, we have 
\[\tilde{L}(a_2) - \tilde{L}(a_2-1) \geq (q-a_2+1)\left(\Floor{\frac{b-1}{a}} -1\right) + 1,\]
which is positive since $q> a_2$ and $b >a$.
\end{proof}
\begin{proposition}\label{prop:cases1and2}
The minimum value of $\tilde{L}$ on $\set{0,\dots, \mu_b}$ is
\[ \tilde{L}(\ell)= \begin{cases}
    q(q-\Floor{\frac{d-1}{a}}) & \text{if } d\le a(q-1),\\
    q-\ell     & \text{if } a(q-1)<d.
\end{cases}\]
\end{proposition}
\begin{proof}
The assumption $d \leq (a+b)(q-1)$ is equivalent to  $\alpha_2=\Floorfrac{d-1-a(q-1)}{b} \leq q-2$. Therefore, by Remark \ref{rem:cases}, $\ell=\max\set{0,\alpha_2}$.

\begin{enumerate}
    \item If $d \leq a(q-1)$ (i.e. $\alpha_2<0$), then the function $\tilde{L}$ is given by Lemma \ref{lem:formula_tildeL} (\ref{it:formula_tildeL1}) and it is strictly increasing on the whole set $\set{0,\dots,\min\set{\Floor{\frac{d-1}{a+b}},q-1}}$ by Lemma  \ref{lem:tildeL_variation}. Thus we clearly have that $\tilde{L}$ reaches its minimum at $\ell=0$.

\item If $a(q-1) < d$, then $\ell=\alpha_2 \geq 0$. The function $\tilde{L}$ is given by Lemma \ref{lem:formula_tildeL} (\ref{it:formula_tildeL2}). Therefore, it follows from Lemma \ref{lem:tildeL_variation} that we only need to compare $\tilde{L}(\ell)=q-\ell$ with 
\[\tilde{L}(\ell+1)=(q-\ell-1)\left(q- \Floor{\frac{d-1-b(\ell+1)}{a}}\right).\]


By definition of the floor function (for more details about this function see \cite{floorfunc}) and $\ell\leq \frac{d-1-a(q-1)}{b}<\ell+1$, one can easily check that 
\[\Floor{\frac{d-1-b(\ell+1)}{a}} < q-1,\]
so $\tilde{L}(\ell+1) > q-\ell-1.$ Then $\tilde{L}(\ell+1)-\tilde{L}(\ell)\geq (q-\ell) - (q-\ell) = 0$.
\end{enumerate}
The proof is complete as the minimum value is always $\tilde{L}(\ell)$.
\end{proof}

\subsection{Minimum of $L$ over the whole set $\red(d)$} \label{sec:lowerbound}

In this part, we give the lower bound provided by Lemma \ref{lem:mindist} by comparing the minimum value $\Tilde{L}(\ell)$ of $\Tilde{L}$ and the minimum of $L$ outside the domain $a_2 \leq \mu_b$.

\subsubsection{The case where $d<bq$}

\begin{proposition}\label{prop:minL_small_degree}
Assume that $(a,b)\neq(1,1)$. If $d < bq$, then 
\[\min_{\a \in \red(d)} L(\a)=\begin{cases}
    q(q-\Floor{\frac{d-1}{a}}) & \text{if } d\le a(q-1),\\
    q-\ell     & \text{if } a(q-1)<d.
\end{cases} \]
\end{proposition}

\begin{proof}
If $b \nmid d$, then for every $\a \in \red(d)$, $a_2 \leq \mu_b$ so the statement follows from Corollary \ref{cor:min_L_and_tildeL} and Proposition \ref{prop:cases1and2}.

If $b \mid d$, we have to compare the minimum provided by Proposition \ref{prop:cases1and2} and the value $L(0,\frac{d}{b})$. Since $d < bq$, we have $q-\frac{d}{b}\geq 1$ and Proposition \ref{prop:func_L} states that

$$L\left(0, \frac{d}{b}\right)=q\left(q-\frac{d}{b}\right)+ \max \left\{q- \Floor{\frac{d-b}{ab}},1\right\}.$$
 
\begin{enumerate}[label=\roman*.]
\item If $d\le a(q-1)$,  then $\alpha_2<0$ and so $\ell=0$. It follows that we have  
$$\frac{d-b}{ab}\le \frac{d-b}{a}\le q-1 \text{ and so } q-\Floor{\frac{d-b}{ab}}\ge 1.$$ 

If $a\nmid d$ , then by Lemma \ref{lem:formula_tildeL} $(1)$, we have 
$$\tilde{L}(0)=q\left(q-\Floor{\frac{d-1}{a}}\right)=q\left(q-\Floor{\frac{d}{a}}\right)\le q\left(q-\frac{d}{b}\right)\le L\left(0,\frac{d}{b}\right).$$ 
If $a\mid d$ and $b>1$, then $b\geq a+1$, $d=d_0ab$ for some integer $d_0\geq 1$ and $d/a-1=d_0b-1\geq d_0a=d/b$. Thus, 
by Lemma \ref{lem:formula_tildeL} $(1)$, we have 
$$\tilde{L}(0)=q\left(q-\Floor{\frac{d-1}{a}}\right)=q\left(q-\Floor{\frac{d}{a}}+1\right)\le q\left(q-\frac{d}{b}\right)\le L\left(0,\frac{d}{b}\right).$$  

Therefore, the minimum of $L$ is $\tilde{L}(0)=q(q-\mu_a)$ (unless  $a=b=1$). 
  
\item If $a(q-1)<d$,  then $\ell=\alpha_2 \geq 0$. By Lemma \ref{lem:formula_tildeL} $(2)$, we have
$$\tilde{L}(\ell)=q-\ell\le q< q+1 \leq L\left(0, \frac{d}{b}\right).$$ 
\end{enumerate}
In all cases, the minimum is the one given by Proposition \ref{prop:cases1and2}.
\end{proof}

\begin{remark}\label{rk:a=b=1}
   If $a=b=1$, let us stress out that 
$$\tilde{L}(0)=q(q-d+1)\ge q(q-d)+q-(d-1)= L\left(0,\frac{d}{b}\right),$$
whereas it is well--known that $d_{min}(C_d)=\tilde{L}(0)=q(q-d+1)$. In this case, the bound provided by Lemma \ref{lem:mindist} is not sharp. However, this issue can be overcome using the 3--transitivity of the projective plane $\Pp^2$ (see \cite[Proposition 4.2]{BDG2019}).
\end{remark}


\subsubsection{The case where $qb \leq d \leq (a+b)(q-1)$}

In this case, we have $a<b \leq a(q-1)$, which induces $q \geq 3$.

Let us first deal with the case $a=1$.

\begin{proposition}\label{prop:minL_a=1}
If $a=1$ and $d \geq bq$, then 
\[\min_{\a \in \red(d)} L(\a)=\begin{cases}
1 & \text{if } b \mid d, \\
q - \ell & \text{otherwise.} 
\end{cases} \]
\end{proposition}
\begin{proof}
If $a=1$ and $d \geq bq$, then every $\a \in \red(d)$ satisfies $a_2 \leq q-1$ except $(0,d/b)$ if $b \mid d$. If $b\nmid d$, the minimum of $L$ on $\red(d)$ is $q-\ell$ as given by Proposition \ref{prop:cases1and2}. If $b\mid d$, then $\Floorfrac{d-b}{ab}=\frac{d}{b}-1 \geq q-1$ and the minimum of $L$ on $\red(d)$ is $1$.
\end{proof}

Now let us assume that $a\geq 2$.

\begin{proposition}\label{prop:minL_ageq2}
Assume that $a\geq 2$ and $d \geq bq$.

If $b \geq a+2$ or $b=a+1$ and $d \geq (a+1)q+\frac{q}{a-1}+1$, then $\min_{\a \in \red(d)} L(\a) = q -\ell$.

Assume $b=a+1$ and $(a+1)q \leq d < (a+1)q+\frac{q}{a-1}+1$.
\begin{itemize}
    \item If $d-(a+1)q \notin \left\langle a,a+1\right\rangle$, then $\min_{\a \in \red(d)} L(\a) = q -\ell$.
    \item If $d-(a+1)q \in \left\langle a,a+1\right\rangle$, set $s=d-\Floorfrac{d}{ab}ab$. Then $\min_{\a \in \red(d)} L(\a)$ is
\[
\begin{cases}
    q- \max\set{\ell, \Floorfrac{d}{ab}-1} & \text{if }  s=0 \: (\text{i.e. } a(a+1)\mid d),\\
    q- \max\set{\ell, \Floorfrac{d}{ab}} & \text{if }  s\notin \left\langle a,a+1\right\rangle \text{ or either } a\mid d \text{ or } a+1 \mid d,\\
    q- \max\set{\ell, \Floorfrac{d}{ab}+1}     & \text{if } s\in \left\langle a,b\right\rangle \text{ with } a \nmid d \text{ and } a+1 \nmid d.
\end{cases}
\]
\end{itemize}
\end{proposition}
\begin{proof}
If $d-bq \notin \left\langle a,b\right\rangle$, then for every $\a \in \red(d)$, we have $a_2 \leq \mu_b=q-1$ and the minimum of $L$ is $\tilde{L}(\ell)=q-\ell$ by Proposition \ref{prop:cases1and2}.

Otherwise, there exists $\a=(a_1,a_2) \in H(d)$ such that $a_2 \geq q$ and the minimum of $L$ is given by the minimum of $\tilde{L}(\ell)=q-\ell$ and $\min \set{L(\a) : \a \in H(d) \text{ such that } a_2 \geq q}$.


For $\a=(a_1,a_2)=(x'-ib,y'+ia) \in H(d)$ such that $a_1 \geq 1$ and $a_2 \geq q$ (which is only possible if $d-a-bq \in \left\langle a,b\right\rangle$), we have 
\[L(a_1,a_2)=q-1-\Floorfrac{a_2-1}{a}=q-1-i,\]
which is a decreasing function of $i$, whose minimum is $q-1-i_{max}$ (see Equation \eqref{eq:imax} for a formula for $i_{\max}$).

If $b \mid d$, then
\[L(0,d/b)=q-\Floorfrac{d-b}{ab}\]
and $(x'-i_{\max}b,y'+i_{\max}a)=(b,d/b-a)$. If $d/b-a < q$, then $(0,d/b)$ is the only point of $H(d)$ such that $a_2 \geq q$. Otherwise, one can easily check $L(b,d/b-a)=L(0,d/b)$, \textit{i.e.} $\Floorfrac{d-b}{ab}=i_{\max}+1$.

We are thus led to compare $\ell$ and $i_{\max}+1$, when $d$ is large enough.

\medskip

First, let us prove that for any $d \geq bq$, we have
    \begin{equation}\label{eq:imax_comparison}
        i_{max}+1 \in  \Floorfrac{d-b}{ab} + \set{0,1}.
    \end{equation}
Writing $d=\lambda ab+ s$ with $0\leq s <ab$, we have 
    \begin{equation}\label{eq:value_at_d/b}
        \Floorfrac{d-b}{ab}=\lambda + \Floorfrac{s-b}{ab}=\begin{cases}
    \lambda -1 &\text{if } s < b, \\
    \lambda &\text{if } s \geq b,
\end{cases}
    \end{equation}
    and using Equation \eqref{eq:imax} gives
    \begin{equation}\label{eq:value_imax}
    i_{max}+1=\lambda + \den(s,a,b) - \mathbf{1}_{a \mid d}- \mathbf{1}_{b \mid d}.    
    \end{equation}
   
   \begin{enumerate}[label=(\roman*)]
       \item\label{it:s<b} If $s < b$, $\den(s,a,b)=1$ if and only if $a$ divides $s$, \textit{i.e.} $\den(s,a,b)=\mathbf{1}_{a \mid d}$. In this case, $i_{max}+1=\lambda - \mathbf{1}_{b \mid d}=\lambda - \mathbf{1}_{s = 0}$.

        \item\label{it:sgeqb} If $s \geq b$, then $\den(s,a,b) - \mathbf{1}_{a \mid d}- \mathbf{1}_{b \mid d} \in \set{0,1}$ because $\mathbf{1}_{a \mid d}= \mathbf{1}_{b \mid d} =1$ if and only if $s=0$. 
     
   \end{enumerate} 
    
\medskip

Secondly, note that $\ell \geq \Floorfrac{d-b}{ab} +1 \geq i_{\max}+1$ when $d \geq \frac{a^2(q-1)+a}{a-1}+b$. It is enough to check that 
    \[d \geq \frac{a^2(q-1)+a}{a-1}+b \: \Leftrightarrow \: \frac{d-1-a(q-1)}{b} \geq \frac{d-b}{ab}+1.\]
    and to use the monotony of the floor function. In this case, $\min_{\a\in\red(d)} L(\a)=q-\ell$.
    
    On the contrary, if $bq \leq d < \frac{a^2(q-1)+a}{a-1}+b$, this forces
  \begin{equation}\label{eq:condition}
      b(q-1)  < \frac{a^2(q-1)+a}{a-1}.
  \end{equation}
  Let us prove that \eqref{eq:condition} can only hold if $b=a+1$, except the case where  $a=q=3$ and $b=5$.

If $a=2$ and $q \geq 3$, the condition in \eqref{eq:condition} can be written 
        \[b(q-1) < 4(q-1)+2 \leq 5(q-1),\]
which implies that $b=3$, because $a$ and $b$ are coprime.

Assume that $a \geq 3$, then isolating $b$ in the inequality \eqref{eq:condition} gives 
    \[b < \frac{a^2}{a-1}+\frac{a}{(a-1)(q-1)}=a+1 + \frac{q+a-1}{(a-1)(q-1)}.\]
    \begin{itemize}
        \item If $a \geq 4$ or $q \geq 4$, then $q+a-1 \leq (q-1)(a-1)$. Then $b < a +2$, hence $b=a+1$.
        \item If $a=q=3$, then a direct computation gives $a<b < \frac{21}{4}$, \textit{i.e.} $b\in\set{4,5}$.
    \end{itemize}

If $a=q=3$ and $b=5$, the only case we have to consider is $d=15$, for which $\ell=1>\Floorfrac{d-b}{ab}=0$.

\medskip

As a result, if $b \geq a+2$, then necessarily $\ell \geq \Floorfrac{d-b}{ab} +1$, which means that the minimum of $L$ over $\red(d)$ is equal to $q-\ell$. When $b=a+1$, we can rewrite \[\frac{a^2(q-1)+a}{a-1}+b=(a+1)q+\frac{q}{a-1}+1.\]
It remains to determine the value of $i_{\max}+1$ in terms of $\lambda$.
In case \ref{it:s<b}, we have $s \leq a$ so $s \in \left\langle a,a+1\right\rangle$ if and only if $s=0$, in which case $i_{\max}+1=\lambda-1$, or $s=a$. If $1 \leq s\leq a$, $i_{\max}+1=\lambda$.
In case \ref{it:sgeqb}, we have $s \geq a+1$ and $\den(s,a,b) - \mathbf{1}_{a \mid d}- \mathbf{1}_{b \mid d}=1$ if and only if $s \in \left\langle a,a+1\right\rangle \text{ with } a \nmid s \text{ and } a+1 \nmid s$. Therefore, we have
\[i_{\max}+1=\begin{cases}
    \lambda-1 & \text{if }  s=0 \: (\text{i.e. } a(a+1)\mid s),\\
    \lambda & \text{if }  s\notin \left\langle a,a+1\right\rangle \text{ or either } a\mid s \text{ or } a+1 \mid s,\\
    \lambda +1  & \text{if } s\in \left\langle a,a+1\right\rangle \text{ with } a \nmid s \text{ and } a+1 \nmid s,
\end{cases}
\]
which completes the proof.
\end{proof}

\section{Polynomials of designated weights}\label{sec:poly_weight}

In order to assert that the lower bound for the minimum distance provided by Propositions \ref{prop:minL_small_degree}, \ref{prop:minL_a=1} and \ref{prop:minL_ageq2} are sharp, it is enough to display a polynomial of the corresponding weight.

\begin{proposition}\label{prop:weight_ell}
 Assume that $b > a$, with $a$ and $b$ coprime. For every degree $d$, there exists a polynomial of weight
 \[\tilde{L}(\ell)=\begin{cases}
    q(q-\Floor{\frac{d-1}{a}}) & \text{if } d\le a(q-1),\\
    q-\ell     & \text{if } a(q-1)<d \le (a+b)(q-1),\\
    1          & \text{if }    d > (a+b)(q-1).
\end{cases}\]
\end{proposition}

\begin{proof}
We provide a polynomial $f$ with $n_f=n-\tilde{L}(\ell)$ roots for each case.\\
\begin{enumerate}[label=(\roman*)]
\item If $d\le a(q-1)$, then $\alpha_2<0$ and $\ell=0$.
We consider the following polynomial,
\[f=x_0^{r+1}\prod_{y_1\in J}(x_1-y_1x_0^a)\] where $r$ is the remainder of the division of $d-1$ by $a$ and $J$ is any subset of $\F_q$ with $|J|=\Floor{\frac{d-1}{a}}$. So, $\deg(f)=r+1+a|J|=d$. The polynomial $f$ has $q+1$ roots with $x_0=0$ and $q|J|$ roots of the form $[1:y_1:y_2]$. In total, it has $q+1+q|J|=q+1+q\Floor{\frac{d-1}{a}}$ roots. Since, 
\[
n-n_f=q^2+q+1-(q+1+q\Floor{\frac{d-1}{a}})=q\left(q-\Floor{\frac{d-1}{a}}\right)=\tilde{L}(0),
\]
the claim follows.
\item If $a(q-1)<d\le a(q-1)+b$,  then $\ell=\alpha_2=0$. The following polynomial
\[f=x_0^{d-(q-1)a}\prod_{y_1\in\F_q^{*}}(x_1-y_1x_0^{a})\] has $(q+1)$ roots with $x_0=0$ and $(q-1)q$ roots of the form $[1:y_1:y_2]$. In total, it has $q^2+1$ roots. Hence, $n-n_f=q^2+q+1-(q^2+1)=q=\tilde{L}(0).$

\item If $a(q-1)+b <d \le (a+b)(q-1)$, then $d-1-a(q-1)=\ell b+r$ with $\ell=\alpha_2\ge 1$ and $0\le r<b$. Let $J'$ be any subset of $\F_q$ with $\size{J'}=\ell$. It is clear that the following polynomial
\[f=x_0^{r+1}\prod_{y_1 \in \F_q^*}(x_1-y_1x_0^a)\prod_{y_2\in J'}(x_2-y_2x_0^b)\]
 has degree $d=r+1+(q-1)a+\ell b$ and $n_f=q^2+\ell+1$, since $f$ has $q+1$ roots with $x_0=0$, $(q-1)(q-\ell)$ roots of the form $[1:y_1:y_2]$ where $y_2\in \F_q\setminus J'$ and $q \ell$ roots of the form $[1:y_1:y_2]$ where $y_2\in J'.$
Thus,
\[n-n_f=q^2+q+1-(q^2+\ell+1)=q-\ell=\tilde{L}(\ell).\]
\item If $(a+b)(q-1) <d$, then $\ell =q-1$. The following polynomial
 \[f=x_0^{d-(q-1)(a+b)}\prod_{y_1\in \F_q^*}(x_1-y_1x_0^{a}) \prod_{y_2\in \F_q^*}(x_2-y_2x_0^{b}) \in S_d
\]
has $q+1$ roots with $x_0=0$, $(q-1)q$ roots of the form $[1:y_1:y_2]$ and $(q-1)$ roots of the form $[1:0:y_2]$ comprising $q^2+q$ roots in total. This means that the weight of the codeword $ev_Y(f)$ is $\tilde{L}(\ell)=1$.
\end{enumerate}
\end{proof}

It remains to deal with polynomials with leading term $\x^\a$ of weight $L(a_1,a_2)$ for $a_2 \geq q$.

\begin{proposition}
Let $a=1$ and $d=d_0b$ with $d_0 \geq q$, then the polynomial
\[f=x_2^{1+d_0-q}\left( x_2^{q-1}-x_0^{b(q-1)} - x_1^{(q-1)(b-1)}\left(x_1^{q-1}-x_0^{q-1}\right)\right) \in S_d\]
satisfies $w(ev_Y(f))=1$.
\end{proposition}

\begin{proof}The polynomial $f$ vanishes everywhere except at $[0:0:1]$. So, $n_f=q^2+q$ and $w(ev_Y(f))=1$.
\end{proof}

The following example shows that polynomials of weight $L(a_1,a_2)$ for $a_2 \geq q$ may not always exist when $a\geq 2$.

\begin{example}\label{ex:(8,2,3,2)}
For $q=8$, $(a,b)=(2,3)$ and $d= 29$, there is exactly one point $\a=(1,9)$ in $\red(d)$ with $a_2> \mu_b=q-1$. The lower bound on the minimum distance provided by Proposition \ref{prop:minL_ageq2} is $L(1,9)=3$. Using \textsc{Magma}, we checked that the actual minimum distance is $q-\ell=4$. This means that no polynomial with leading term $x_1x_2^9$ has $q^2+q-2$ zeroes.
\end{example}

Therefore, as in the case $a=b=1$ (see Remark \ref{rk:a=b=1}), the footprint bound may not provide the actual minimum distance. For the purpose of finding the minimum distance, it is enough to deal with the cases where $\min_{\a \in \red(d)} L(\a) \neq q-\ell$. 

Polynomials of the form $f=g_1\left(x_2^q-x_0^{b(q-1)}x_2\right)+g_2\left(x_1^q-x_0^{a(q-1)}x_1\right)$, as provided by Proposition \ref{prop:form}, clearly vanish on the domain $x_0 \neq 0$. In order to check if the lower bound of Proposition \ref{prop:minL_ageq2} is reached, we need to count the number of zeroes of polynomials of this form on the line $x_0=0$.

\begin{remark}\label{rk:rat_pts}
The $\F_q$-rational points on the line $x_0 = 0$ consist of the following: the two special points $[0:1:0]$ and $[0:0:1]$ and the points of the form $[0:y_1:y_2]$ for which the ratio $y_2^a/y_1^b$ belongs to $\F_q^*$. Among these, there are exactly $q - 1$ such points of the second type, since each nonzero element of $\F_q^*$ uniquely determines a value of the ratio $y_2^a/y_1^b$, and hence a corresponding point $[0:y_1:y_2]$ up to scalar equivalence. We now analyze when we can simplify the representatives by assuming $y_1 = 1$ and $y_2 \in \F_q^*$. This is always possible if $\gcd(a, q - 1) = 1$, because in this case the map $y_2 \mapsto y_2^a$ is a bijection on $\F_q^*$. Therefore, for every value of $y_2^a/y_1^b \in \F_q^*$, we can find a unique $y_2 \in \F_q^*$ such that $[0:1:y_2]$ represents the point. However, if $\gcd(a, q - 1) \ne 1$, then this simplification is only possible for certain points. Specifically, let $\eta$ be a generator of the multiplicative group $\F_q^*$. Then the equation $y_2^a = \eta^j$ has a solution $y_2 = \eta^{j_2} \in \F_q^*$ if and only if the congruence \[a j_2 \equiv j \pmod{q - 1}\] has a solution. This happens if and only if $\gcd(a, q - 1)$ divides $j$. Consequently, the number of such solvable $j$ values (and hence the number of such points in the form $[0:1:\eta^i]$) is exactly\[\frac{q - 1}{\gcd(a, q - 1)}.\] So, the points for which a representative of the form $[0:1:\eta^i]$ exists are precisely those with $i = 1, \dots, (q -1)/\gcd(a, q - 1)$.
\end{remark}

\begin{proposition}\label{prop:a_divides_q-1}
Assume that $a \mid q-1$ and $b=a+1$. Let $\a=(a_1,a_2) \in H(d)$ with $a_2 \geq q$.
Write $a_2=q+k_2a+r_2$ with $0 \leq r_2 \leq a-1$.
Write $\F_q^*=\left\langle \eta \right\rangle$ and take $J \subset \set{1,\dots,q-1}  \setminus \set{ia : i \in \set{1,\dots,(q-1)/a}}$  of cardinality $k_2$.

\[f=\left(\prod_{j \in J} (x_2^a-\eta^j x_1^b)\right)x_1^{a_1}x_2^{1+r_2}\left( x_2^{q-1}-x_0^{b(q-1)} - x_1^{(q-1)/a}\left(x_1^{q-1}-x_0^{a(q-1)}\right)\right) \]
has weight $L(a_1,a_2)$.
\end{proposition}

\begin{proof}Since $a_2 \geq q$, Proposition \ref{prop:func_L} says $f$ has at most $n-L(a_1,a_2)$ roots, where
 \[L(a_1,a_2)=q-\mathds{1}_{a_1 > 0}-\Floorfrac{q+k_2a+r_2-1}{a}=q-\mathds{1}_{a_1 > 0} - k_2 - \frac{q-1}{a}.\]
 Notice that if $d \notin \reg(Y)$, then
\[k_2 + \frac{q-1}{a}=\Floorfrac{d-bq-a a_1+b(q-1)}{ab}=\Floorfrac{d-b-a a_1}{ab}\leq\Floorfrac{d}{ab} \leq q-1.\]


Clearly, the associated polynomial $f$ vanishes at $[1:y_1:y_2]$ for every $y_1,y_2\in \F_q$. On the line $x_0=0$, we have
 \[f(0,x_1,x_2)=\left(\prod_{j \in J} (x_2^a-\eta^j x_1^b)\right)x_1^{a_1}x_2^{1+r_2}\left( x_2^{q-1}-x_1^{(q-1)/a}x_1^{q-1}\right).
\]
Obviously, $f$ vanishes at $[0:1:0]$. It vanishes at $[0:0:1]$ if and only if $a_1>0$. The first factor of $f$ vanishes at the $k_2$ points $[0:y_1:y_2]$ such that $y_2^a/y_1^b=\eta^j$ for $j \in J$.
The last factor of $f$ vanishes at the $(q-1)/a$ points $[0:1:\eta^{i}]$ for $i\in \set{1,\dots,(q-1)/a}$. Since $[0:1:\eta^i]=[0:1:\eta^{i'}] \iff q-1 \text{ divides }a(i'-i)$, these points are all distinct. As a result, we have $w(\ev(f))=L(a_1,a_2)$.
\end{proof}

\section{Conclusion and open questions}\label{sec:conclusion}

\begin{theorem}\label{thm:minimum_distance}
Assume that $b\geq a$ are coprime.

If $a=1$, then
\[
d_{min}(C_d)=\begin{cases}
    q(q-d+1) & \text{if } d\le q-1,\\
    q-\Floorfrac{d-q}{b} & \text{if } q \le d < bq,\\
    1           & \text{if } d \geq bq.
\end{cases}
\]

Assume $a \geq 2$. Set $\ell=\Floorfrac{d-1-a(q-1)}{b}$. 

If $b \geq a+2$, then
\[
d_{min}(C_d)=\begin{cases}
    q(q-\Floorfrac{d-1}{a}) & \text{if } d\le a(q-1),\\
    q-\ell & \text{if } a(q-1) < d \leq (a+b)(q-1),\\
    1           & \text{if } d > (a+b)(q-1).
\end{cases}
\]

If $b=a+1$, then

\[
d_{min}(C_d)=\begin{cases}
    q(q-\Floorfrac{d-1}{a}) & \text{if } d\le a(q-1),\\
    q-\ell & \text{if } a(q-1) < d < bq \text{ or } d \geq bq+\frac{q}{a-1}+1,\\
    1           & \text{if } d > (a+b)(q-1).
\end{cases}
\]

If $b=a+1$ and $bq \leq d < bq+\frac{q}{a-1}+1$, several cases have to be distinguished, according to whether $d-(a+1)q$ lies in the semigroup $\left\langle a,a+1\right\rangle_{\N}$ generated by $a$ and $b=a+1$.
        \begin{itemize}
    \item If $d-(a+1)q \notin \left\langle a,a+1\right\rangle_{\N}$, then $d_{min}(C_d) = q -\ell$.
    \item If $d-(a+1)q \in \left\langle a,a+1\right\rangle_{\N}$, set $s=d-\Floorfrac{d}{ab}ab$. Then $d_{min}(C_d)$ is \emph{bounded from below} by
\[
\begin{cases}
    q- \max\set{\ell, \Floorfrac{d}{ab}-1} & \text{if }  s=0 \: (\text{i.e. } a(a+1)\mid d),\\
    q- \max\set{\ell, \Floorfrac{d}{ab}} & \text{if }  s\notin \left\langle a,a+1\right\rangle \text{ or } a\mid s \text{ or } a+1 \mid s,\\
    q- \max\set{\ell, \Floorfrac{d}{ab}+1}     & \text{if } s\in \left\langle a,b\right\rangle \text{ with } a \nmid s \text{ and } a+1 \nmid s.
\end{cases}
\]

The equality holds when the value is $q-\ell$ or when $a \mid q-1$.

\end{itemize}

\end{theorem}
\begin{proof}
The minimum distance and the minimum weight codewords when $a=b=1$ are provided by \cite[Corollary 3.3 \& Theorem 4.3]{GL2023}. Let us assume that $b>a$.

For degrees $d < bq$, the inequality $d_{min}(C_d)\geq \tilde{L}(\ell)$ follows from Proposition \ref{prop:minL_small_degree}. Proposition \ref{prop:weight_ell} yields the reverse inequality.

For degrees $d \geq bq$, the lower bound for the minimum distance is given by Proposition \ref{prop:minL_ageq2}. When the bound is $q-\ell$, the equality follows from Proposition \ref{prop:weight_ell}. Proposition \ref{prop:a_divides_q-1} guarantees the equality when $a \mid q-1$.
\end{proof}

\begin{remark}
   Proposition \ref{prop:form} forces polynomials $f$ with weight $w(\ev(f)) \neq q-\ell$ that reach the minimum distance to vanish on the affine plane $x_0 \neq 0$. The lower bound displayed in Theorem \ref{thm:minimum_distance} could also be obtained by factorizing $f(0,x_1,x_2)$:
   \[f(0,x_1,x_2)=h(x_1,x_2) \cdot  \prod_{i=1}^j(x_2^a-\alpha_i x_1^b)\]
   for some $j \leq \Floorfrac{d}{ab}$, $\alpha_1,\dots,\alpha_j \in \F_q$ and where $h$ is not divible by any factor of the form $(x_2^a-\alpha x_1^b)$.
   In this case, $f$ has either $j$, $j+1$ or $j+2$ roots, depending on whether $x_1$ or $x_2$ divide $h$. Indeed,  if $a \nmid d$ (respectively $b \nmid d)$, then $x_2$ (respectively $x_1$) divides $h$ (hence $f(0,x_1,x_2)$), which implies that $f$ vanishes at $[0:1:0]$ (resp.$[0:0:1]$). 
      
   Finally, one can easily check that if $s \notin \left\langle a,b\right\rangle$, then $j \leq \Floorfrac{d}{ab}-1$ and $f$ is divisible by $x_1$ and $x_2$, which yields at most $\Floorfrac{d}{ab}+1$ roots.
   
\end{remark}

\subsection{Refining the lower bound when $b=a+1$ and $bq \leq d < bq+\frac{q}{a-1}+1$}

In this paragraph, we focus on the case where Theorem \ref{thm:minimum_distance} only gives a lower bound on the minimum distance, because we are not able to guarantee the existence of polynomials of a specified weight.
Note that if $w(\ev(f))= L(a_1,a_2)$ then, by Lemma \ref{lem:U}, the union set $\cG \cup \set{f}=\set{f_0,f_1,f_2,f}$ is a Gr\"obner basis of $I(Y,f)$ (see Theorem \ref{thm:ideal} for the definition of $\cG$). In particular, this means that
${\overline{S(f,f_1)}^{\, \cG \cup \set{f}}=0}$. The next proposition characterizes (modulo $I(Y)$) polynomials with such a property.

\begin{proposition}\label{prop:form} Assume that $b > a\geq 2$, with $a$ and $b$ coprime.
Let $d \geq bq$ and $f \in S_d \setminus I(Y)$ such that $\lm(f)=x_1^{a_1}x_2^{a_2}$ with $a_2 \geq q$.
If $\overline{S(f,f_1)}^{\, \cG \cup \set{f}}=0$, which is the case when $w(\ev(f))= L(a_1,a_2)$, there exists $g_1 \in S_{d-bq}$ with $\lm(g_1)=x_1^{a_1}x_2^{a_2-q}$ and $g_2 \in S_{d-aq}$ such that \[f =g_1\left(x_2^q-x_0^{b(q-1)}x_2\right)+g_2\left(x_1^q-x_0^{a(q-1)}x_1\right) \text{ modulo } I(Y).\]
\end{proposition}

\begin{proof}
Assume that the leading coefficient of $f$ is $1$ and write $f=x_1^{a_1}x_2^{a_2}+g$ with $a_2 \geq q$ and $\lm(g) < x_1^{a_1}x_2^{a_2}.$ On one hand, we have
\begin{equation}\label{eq:S(f,f1)}
S(f,f_1)= x_0f-x_1^{a_1}x_2^{a_2-q}f_1=x_0g + x_0^{(q-1)b+1}x_1^{a_1}x_2^{a_2-q+1}.    
\end{equation}

On the other hand, as $\cG \cup \set{f}$ is a Gr\"obner basis of $I(Y,f)$, the division algorithm enables us to write 
\[
S(f,f_1)=g_0f_0+g_1f_1+g_2f_2+g_3f\]
with some homogeneous polynomials $g_0,\:g_1,\:g_2$ and $g_3$, such that that the leading monomials of each term $g_if_i$ and $g_3f$ are smaller than or equal to $\lm(S(f,f_1))$. As $\deg S(f,f_1)= \deg f +1$, we have either $\deg g_3=1$ or $g_3=0$. The $S$-polynomial satisfies $\lm(S(f,f_1))< \lcm(\lm(f),\lm(f_1))=x_0x_1^{a_1}x_2^{a_2}$, so having $\lm(g_3f) \leq \lm(S(f,f_1))$ forces $g_3=0$ as $S_1=\Span\{x_0\}$.

Noticing that $f_0$ does not depend on $x_0$, we get
    \[g =  -x_0^{(q-1)b}x_1^{a_1}x_2^{a_2-q+1}+(g_0/x_0)f_0+g_1(f_1/x_0)+g_2(f_2/x_0). \]
    Modulo $\cG$, we can assume without loss of generality that $g_0=0$. Then
    \begin{align*}
    f&=x_1^{a_1}x_2^{a_2}-x_0^{(q-1)b}x_1^{a_1}x_2^{a_2-q+1}+g_1(f_1/x_0)+g_2(f_2/x_0)\\
    &=(x_1^{a_1}x_2^{a_2-q}+g_1)f_1/x_0 + g_2 f_2/x_0,
    \end{align*}
    completing the proof.
\end{proof}

\begin{remark}\label{rem:form}
   Modulo $I(Y)$, we can assume that the polynomials $g_1$ and $g_2$ do not involve the variable $x_0$. However, the resulting polynomial $f$ may not be reduced.
   For instance take a reduced  polynomial $f=x_1\left(x_1^{2(q-1)}-x_0^{2a(q-1)}\right) \in \Span(\overline{\M}_d)$. We can write $f=\tilde{g}_2(x_1^q-x_0^{a(q-1)}x_1)$ with $\tilde{g}_2=x_ 1^{q-1}+x_0^{a(q-1)}$. Modulo $f_2$, we have $f  \equiv x_1^{2q-1}-x_0^{a(q-1)}x_1^q=g_2(x_1^q-x_0^{a(q-1)}x_1)$ with $g_2=x_1^{q-1}$.

\end{remark}

\begin{proposition}\label{prop:Delta'}
Let $f\in S_d \setminus I_d(Y)$ such that $\lm(f)=\x^{\a,d}$ with $\a \in \red(d)$ and $a_2 \geq q$. Assume that
$\overline{S(f,f_1)}^{\, \cG\cup \{f\}}\neq 0$ where $f_1=x_{0}x_{2}^q-x_0^{(q-1)b+1}x_2$ is as in Theorem \ref{thm:ideal}. Then $w(\ev_Y(f)) \geq q-\ell$.
\end{proposition}
\begin{proof}

For $\Tid \geq d$ we define
\[\overline{\Delta}'_{\Tid}(f)= \set{ M \in \overline{\M}_{\Tid}  : \lm(f) \nmid M \text{ and }  \lm(\overline{S(f,f_1)}^{\, \cG\cup \{f\}}) \nmid M }. \]
A similar argument as in Lemma \ref{lem:U} leads to
\begin{equation}\label{eq:nf_Delta'}
\Tilde{n}_f(\Tid)=H_{I(Y,f)}(\Tid) \leq \size{\overline{\Delta}'_{\Tid}(f)}.
\end{equation}

We can express $\overline{\Delta}'_{\Tid}(f)$ in terms of projective shadows as follows:
\[
\overline{\Delta}'_{\Tid}(f) =  \overline{\M}_{\Tid} \setminus 
\left(\overline{\nabla}_{\Tid}(\lm(f)) \cup \overline{\nabla}_{\Tid}(\lm(\overline{S(f,f_1)}^{\, \cG\cup \{f\}}) \right).
\]

\medskip

Note that no monomial $\x^{\tilde{\a},\Tid} \in \overline{\M}_{\Tid}$ is divisible by both $\lm(f)=x_1^{a_1}x_2^{a_2}$ and $\lm(\overline{S(f,f_1)}^{\, \cG\cup \{f\}}$, \textit{i.e.}
\[\overline{\nabla}_{\Tid}(\lm(f)) \cap \overline{\nabla}_{\Tid}(\lm(\overline{S(f,f_1)}^{\, \cG\cup \{f\}}) = \varnothing.\]
Indeed, as $a_2 \geq q$ and $x_0 \mid \lm(\overline{S(f,f_1)}^{\, \cG\cup \{f\}}$ (see Equation \eqref{eq:S(f,f1)}), such a monomial would be divisible by $x_0x_2^q=\lm(f_1)$, which is impossible by definition of $\overline{\M}_{\Tid}$. This means that\[
\size{\overline{\Delta}'_{\Tid}(f)}=n-
\size{\overline{\nabla}_{\Tid}(\lm(f))} - \size{\overline{\nabla}_{\Tid}(\lm(\overline{S(f,f_1)}^{\, \cG\cup \{f\}})}.
\]

Write $\lm(\overline{S(f,f_1)}^{\, \cG\cup \{f\}})=\x^{\a',d+1}$. Provided that $\Tid$ is large enough, Equation \eqref{eq:nf_Delta'} can be rewritten
\begin{align*}w(\ev_Y(f))\geq n-\size{\overline{\Delta}'_{\Tid}(f)}&= \size{\overline{\nabla}_{\Tid}\left(\lm(f)\right)} + \size{\overline{\nabla}_{\Tid}\left(\lm(\overline{S(f,f_1)}^{\, \cG\cup \{f\}})\right)} \\
&= L(\a) + L(\a').
\end{align*}
On one hand, we have $L(\a) \geq 1$ (see Remark \ref{rem:L>0}). On the other hand, note that $\a' \in \red(d+1) \setminus H(d+1)\subseteq \red(d+1) \cap \set{a_2 \leq \mu_b}$ and we can apply Proposition \ref{prop:cases1and2} on degree $d+1$ to bound $L(\a')$ from below. Gathering up, we get
\[w(\ev_Y(f))\geq1+L(\a') \geq 1+q-\Floorfrac{d-a(q-1)}{b} \geq q-\ell,
\]
which proves the statement.
\end{proof}

If one wants to compute exactly the minimum distance in cases for which Theorem \ref{thm:minimum_distance} only supplies for a lower bound, previous results indicate that the study can be restricted to a smaller code.

\begin{corollary}\label{cor:smaller_code} Let $d \geq bq$. Denote by $Z \subset \Pp(1,a,b)$ the set formed by the $q+1$ $\F_q$-points on the line $\set{x_0=0}$. Let us consider the code
\begin{align*}\tilde{C}_{d,Z}=&\Span\left(\set{\ev_Z(\x^{\a,d-bq} x_2^{q}) : \a \in H(d-bq)}\right) \\
&+ \Span\left( \set{\ev_Z(\x^{\a,d-aq} x_1^{q}) : \a \in H(d-aq)}\right)
\end{align*}
The code $\tilde{C}_{d,Z}$ has length $q+1$ and its dimension satisfies 
\[\dim(\tilde{C}_{d,Z})\leq \den(d-aq,a,b)\cdot\den(d-bq,a,b).\]
Moreover, we have $d_{min}(C_d)=\min(q-\ell, d_{min}(\tilde{C}_{d,Z}))$.
\end{corollary}

\begin{proof}The dimension of $\tilde{C}_{d,Z}$ follows from Equation \eqref{eq:H(d)-rephrase}. We separate $C_{Y,d}$ in three sets: $C_{Y,d}=\ev_{Y,d}(S_d^{(1)} \cup S_d^{(2)} \cup S_d^{(3)})$ where
\begin{align*}
    S_d^{(1)}&=\set{ f\in S_d \text{ such that }\lm(f)=\x^{\a,d} \text{ with } a_2 < q},\\
    S_d^{(2)}&=\set{ f\in S_d \text{ such that }\lm(f)=\x^{\a,d} \text{ with } a_2 \geq q \text{ and } \overline{S(f,f_1)}^{\, \cG\cup \{f\}}\neq 0},\\
    S_d^{(3)}&=\set{ f\in S_d \text{ such that }\lm(f)=\x^{\a,d} \text{ with } a_2 \geq q \text{ and } \overline{S(f,f_1)}^{\, \cG\cup \{f\}}= 0}.
\end{align*}
Then $d_{min}(C_{Y,d})=\min\set{\min\set{w(\ev_{Y}(f)) : f \in S_d^{(i)}} : i \in \set{1,2,3} }$.

As in Theorem \ref{thm:minimum_distance}, $\min\set{w(\ev_{Y}(f)) : f \in S_d^{(1)}} = q-\ell$. Proposition \ref{prop:Delta'} ensures that $w(\ev_Y(f)) \geq q-\ell$ for $f \in S_d^{(2)}$. If $f \in S_d^{(3)}$, we can assume that $f =g_1\left(x_2^q-x_0^{b(q-1)}x_2\right)+g_2\left(x_1^q-x_0^{a(q-1)}x_1\right)$ modulo $I(Y)$ by Proposition \ref{prop:form}. In this case $f$ is zero on $Y \setminus Z$ so $w(\ev_Y(f))=w(\ev_Z(g_1x_2^q+g_2x_1^q))$. We deduce from Remark \ref{rem:form} that $\ev_{Y}\left(S_d^{(3)}\right) = \tilde{C}_{d,Z}$, which concludes the proof.
\end{proof}

\begin{example} \label{eg:the last} Let $q=16$, $X=\P(1,2,3)$, $Y=X(\F_q)$ and $d=48=bq$. We have $\ell=5$, hence $q-\ell=16-5=11$. Since $d-(a+1)q\in \langle a,b\rangle$ and  $ab \mid d$, then $s=0$. Theorem \ref{thm:minimum_distance} gives $d_{min}(C_d) \geq q-\max\{\ell, \Floorfrac{d}{ab}-1\}=\min\set{11,9}=9$. But we checked using Magma  \cite{magma} that the actual minimum distance is $11$. This also follows from Corollary \ref{cor:smaller_code} as the minimum distance of the smaller code $\tilde{C}_{d,Z}$ is $13>11$.
\end{example}

\section*{Acknowledgements} 

The second author expresses her deep gratitude to Rodrigo San José Rubio for some discussions that revealed some mistakes in the early version of the present work. The authors thank the anonymous referees for their valuable questions and suggestions.

	\bibliographystyle{plainnat}      
	\bibliography{ref}   
\end{document}